\documentclass[11pt]{article}

\usepackage[ansinew]{inputenc}
 
\usepackage{amssymb,amsmath,amstext,amscd,amsthm,amsfonts}
\usepackage{mathtools, relsize}
\usepackage{psfrag}
\usepackage{graphicx}
\usepackage{latexsym}
\usepackage{dsfont}
\usepackage{nicefrac}

\usepackage{textfit}

\usepackage[all]{xy}

\input{liemacs-own.sty}

\newcounter{thmcount}[section]
\renewcommand{\thethmcount}{\arabic{section}.\arabic{thmcount}}
\renewcommand{\theequation}{\arabic{section}.\arabic{equation}}

\newtheorem{thm}[thmcount]{Theorem}
\newtheorem{prop}[thmcount]{Proposition}
\newtheorem{lem}[thmcount]{Lemma}
\newtheorem{cor}[thmcount]{Corollary}

\theoremstyle{definition}

\newtheorem{rmk}[thmcount]{Remark}
\newtheorem{dfn}[thmcount]{Definition}

\def \H {\mathcal{H}}

\def \cA {\mathcal{A}}

\def \P {\mathbb{P}}
\def \D {\mathcal{D}}
\def \t {\mathfrak{t}}
\def \Re {\mathrm{Re}}
\def \Im {\mathrm{Im}}

\def \id {\mathrm{Id}}
\def \spn {\mathrm{span}}

\def \vp {\varphi}

\usepackage{fancyhdr}
\begin{document} 

\title{Complex Semigroups for Oscillator Groups}
\author{
Christoph Zellner
\begin{footnote}
{Department Mathematik,
FAU Erlangen-N\"urnberg,
Cauerstra\ss e 11, 91058 Erlangen, Deutschland,
\texttt{zellner@mi.uni-erlangen.de}
}\end{footnote}
\begin{footnote}{Supported by DFG-grant NE 413/7-2, Schwerpunktprogramm 
``Darstellungstheorie''.}\end{footnote} 
}

\maketitle
\date 


\begin{abstract} 
An oscillator group $G$ is a semidirect product of a Heisenberg group with a one-parameter group. In this article we construct Olshanski semigroups for infinite-dimensional oscillator groups. These are complex involutive semigroups which have a polar decomposition. The main application will be for representations $\pi$ of $G$ which are semibounded, i.e., there exists a non-empty open subset $U$ of the Lie algebra $\g$ such that the operators $i\dd\pi(x)$ from the derived representation are uniformly bounded from above for $x\in U$. More precisely we show that every semibounded representation of an oscillator group $G$ extends to a non-degenerate holomorphic representation of such a semigroup and conversely each non-degenerate holomorphic representation of such a semigroup gives rise to a semibounded representation of $G$. The main application of this result is a classification of representations of the canonical commutation relations with a positive Hamiltonian, which will be obtained in a subsequent paper. Moreover it yields direct integral decomposition into irreducible ones and implies the existence of a dense subspace of analytic vectors for semibounded representations of $G$.
\end{abstract}

\tableofcontents

\section{Introduction}

For a finite-dimensional Lie group $G$ and an invariant cone $W$ in its Lie algebra $\g$ (cones are assumed to be convex), such that the spectrum of the operators $\ad x$ is purely imaginary for $x$ in the algebraic interior of $W$, one can associate a corresponding Olshanski semigroup $\Gamma_G(W)$ (Lawson's Theorem). It is a complex involutive semigroup and has a polar decomposition $\Gamma_G(W)=G\exp(iW)$, cf. \cite{Ne00}. For infinite-dimensional Lie groups there is no general analogous result, even not for Banach--Lie groups. However, there are natural examples of Olshanski semigroups which arise as compression semigroups of bounded symmetric domains in complex Banach spaces, symmetric Hilbert domains, and real forms of such domains (\cite{Ne01}). In this paper we address the problem of the existence of Olshanski semigroups $\Gamma_G(W)$ when $G$ belongs to the class of infinite-dimensional oscillator groups described below. The main goal is to obtain a better understanding of the canonical commutation relations with a positive Hamiltonian which is achieved in a subsequent paper using holomorphic extensions of the corresponding representations to the complex semigroups associated to oscillator groups.

Let $(V,\omega)$ be a locally convex symplectic vector space and $\gamma:\R\rightarrow{\rm Sp}(V,\omega)$ be a one-parameter group of symplectomorphisms defining a smooth action of $\R$ on $V$ and denote by $D:=\gamma'(0)$ its generator. Then the Lie group
$$G(V,\omega,\gamma):=\Heis(V,\omega) \rtimes_\gamma\R$$
is called an \textit{oscillator group}, where $\Heis(V,\omega):=\R\times_\omega V$ is the \textit{Heisenberg group} with multiplication given by
$$ (t,x)(s,y) = \big(t+s+\textstyle{\frac{1}{2}}\omega(x,y) ,x+y\big).$$
The Lie algebra of $G(V,\omega,\gamma)$ is $\g(V,\omega,\gamma)=\heis(V,\omega)\rtimes_D \R$ with the bracket
$$[(t,v,s),(t',v',s')]=(\omega(v,v'),sDv'-s'Dv,0).$$
There is a particular interesting class of oscillator groups $G_A:=\Heis(V_A,\omega_A) \rtimes_{\gamma_A} \R$ labeled by injective non-negative operators $A$ on complex Hilbert spaces $H_A$, where $V_A=C^{\infty}(A)\subset H_A$ is the domain of smooth vectors for $A$, $\omega_A(x,y)=\operatorname{Im} \langle Ax,y\rangle$ and $\gamma_A(t) = e^{itA}$. In particular, the symplectic form $\omega_A$ is induced by a complex structure on a pre-Hilbert space $V_A$, which is not always the case for a general symplectic space $(V,\omega)$, see \cite{Ro93}. The groups $G_A$ are called \textit{standard oscillator groups}.

Each $G_A$ contains a canonical dense subgroup $G_A^\cO$ which has a complexification $G_{A,\C}$. For every non-empty open invariant cone $W$ in $G_A$ we construct a complex semigroup $S_W\subset G_{A,\C}$ and show that it is a complex involutive semigroup and has a polar decomposition. The main application will be for semibounded representations of $G_A$.

Let $G$ be a Lie group with an exponential map and $\pi:G\rightarrow \U(\H)$ be a unitary representation which is smooth, i.e., the space of smooth vectors 
$$\H^\infty:=\{v\in \H : G\to\H, g\mapsto \pi(g)v\text{ is smooth}\}$$
is dense in $\H$. For a smooth representation $\pi:G\rightarrow \U(H)$ we define the \textit{derived representation}
$$\dd\pi:\g\rightarrow \End(\H^\infty),\quad  \dd\pi(x)v:= \derat0 \pi(\exp(tx))v,$$
which is a representation of the Lie algebra $\g$ of $G$ by essentially skew-adjoint operators on $\H^\infty$. We call $\pi$ \textit{semibounded} if the \textit{support function}
$$s_\pi : \g \rightarrow \R \cup \{\infty\}, \quad s_\pi(x):= \sup(\Spec(i\dd\pi(x)))$$
is bounded on a non-empty open subset of $\g$. The basic concept of semibounded representations was developed in \cite{Ne09, Ne08}. One important point is that $s_\pi$ can be expressed in terms of the \textit{momentum set} $I_\pi\subset \g'$, where $\g'$ denotes the topological dual of $\g$. It is defined as the weak-*-closed convex hull of the image of the \textit{momentum map}
$$\Phi_\pi : \P(\H^\infty)\rightarrow \g', \quad \Phi_\pi([v])(x):= \frac{1}{i}\frac{\langle\dd\pi(x)v,v\rangle}{\langle v,v\rangle}.$$
Then we have $s_\pi(x) = - \inf \langle I_\pi,x \rangle, x\in \g$ and set $B(I_\pi):=\{x\in\g : s_\pi(x) <\infty\}$.

Generalizing the main result of \cite{MN11} which only applies to Banach--Lie groups we show that every semibounded representation of a simply-connected Mackey-complete Lie group extends to a holomorphic representation of a complex semigroup, provided this semigroup exists. Specializing to oscillator groups this implies that every semibounded representation of $G_A$ extends to a holomorphic representation of some $S_W$, which allows us to construct $C^*$-algebras associated to $S_W$ whose representations correspond to semibounded representations of $G_A$. This in turn implies the existence of direct integral decomposition of semibounded representations into irreducible ones in the separable case. Furthermore it plays a central role in the classification of representations of the canonical commutation relations with a positive Hamiltonian, which will be obtained in \cite{Ze15}, cf. also \cite{Ze14}. Other $C^*$-algebras whose representations correspond to the regular representations of the canonical commutation relations were constructed in \cite{GN09}. For a detailed discussion of complex semigroups and holomorphic extensions of representations of finite-dimensional oscillator groups we refer to \cite{HO92}.

Now we give an outline of this paper. In Section 2 we classify the open invariant cones in an oscillator algebra $\g:=\g(V,\omega,\gamma)$ under the mild condition that $DV\subset V$ is dense and consider some basic properties of semibounded representations $\pi$ of $G(V,\omega,\gamma)$. In particular we give a formula for $s_\pi$. Under the assumption that $DV\subset V$ is dense and $G(V,\omega,\gamma)$ embeds into a standard oscillator group $G_A$, cf. \cite{NZ13}, we show that every semibounded representation $\pi$ of $G(V,\omega,\gamma)$ extends to a semibounded representation of $G_A$. In particular, $\pi$ is automatically smooth with respect to the coarser subspace topology on $G(V,\omega,\gamma)\subset G_A$. Therefore the semibounded representations of standard oscillator groups essentially cover those of general oscillator groups.

Section 3 is devoted to the construction of the Olshanski semigroups corresponding to non-empty open invariant cones $W$ in the Lie algebra $\g_A$ of $G_A$. This is done in the following way: We consider the dense subgroup $G_A^{\cO}:=\Heis(V^\cO,\omega_A) \rtimes_{\gamma_A} \R$ of $G_A$ where $V^\cO \subset V_A$ denotes the space of $\gamma_A$-holomorphic vectors. Here $V^\cO$ is equipped with a natural topology which is finer than the $C^\infty$-topology on $V_A$. Then $G_A^\cO$ is a Fr\'echet-Lie group and admits a complexification $G_{A,\C}$.  For each non-empty open invariant cone $W$ we show that the subset $S_W=G_A^\cO\cdot\exp_\C(i(W\cap \g_A^\cO))$ of $G_{A,\C}$ is an open subsemigroup. Moreover we show that $S_W$ admits an involution and a polar decomposition, i.e., the map 
$$G_A^\cO\times (W\cap \g_A^\cO) \rightarrow S_W,(g,w)\mapsto g\exp_\C(iw)$$
is a diffeomorphism, whose inverse turns out to be analytic if $W$ is a proper cone. 

In Section 4 we consider holomorphic extensions of semibounded representation of general Lie groups, where we assume that the corresponding complex semigroups exist. More precisely we show that every semibounded representation of a simply-connected Mackey complete Lie group extends to a holomorphic representation of such a complex semigroup. This generalizes a result from \cite{MN11} for Banach--Lie groups to Mackey complete Lie groups. Furthermore we prove, under natural sufficient conditions, the extendability of smooth resp. semibounded representations of dense subgroups. This will be used in Section 5 to show that a semibounded representation $\pi$ of $G_A$ extends to non-degenerate holomorphic representation of the complex semigroup corresponding to the cone $B(I_\pi)$, which implies that $\pi$ has a dense space of analytic vectors. Moreover we show that for a non-empty proper open invariant cone $W\subset\g_A$ there is a one-to-one correspondence between semibounded representations $\pi$ of $G_A$ satisfying $W\subset B(I_\pi)^0$ and non-degenerate holomorphic representations of $S_W$ which preserves commutants. From \cite{Ne08} it is known that the non-degenerate holomorphic representations of $S_W$ correspond to representations of certain $C^*$-algebras. This will allow us to apply $C^*$-methods to semibounded representations of $G_A$. 

A result on the existence of complex semigroups associated to elliptic open invariant cones $W$ in a general Banach--Lie group $G$ would be desirable, where $W$ is called elliptic if $e^{\R\ad x} \subset \End(\g)$ is equicontinuous for all $x\in W$. This has not been obtained so far. However an alternative approach avoiding complex semigroups is established in \cite{NSZ15}, which makes $C^*$-methods available for semibounded representations of any metrizable Lie group yielding in particular direct integral decomposition of semibounded representations into irreducible ones. Moreover a refinement of some methods used from \cite{MN11} shows that every semibounded representation $\pi$ of a Banach--Lie group $G$ extends locally to a holomorphic map in the following sense: There exists an open $0$--neighborhood $U$ in $\g$ and a complex manifold structure on $U\times (U\cap W_\pi)$ such that the map
$U\times (U\cap W_\pi)\to B(\cH), (x,w)\mapsto \pi(\exp(x))e^{i\dd\pi(w)}$ is holomorphic. Now one can consider the $C^*$-algebra $\cA_\pi\subset B(\cH)$ generated by $\pi(G)e^{i\dd\pi(W_\pi)}$ and conclude that every non-degenerate representation $\rho$ of $\cA_\pi$ yields a semibounded representation $\rho_G$ of $G$ determined by $\rho(\pi(g)a)=\rho_G(g)\rho(a), g\in G,a\in \cA_{\pi}$. This also makes $C^*$-methods available for semibounded representations of Banach--Lie groups and moreover yields a dense subspace of analytic vectors for such.

\section{Oscillator groups}\label{sectog}

\subsection{The exponential function of an oscillator group}

We always want the oscillator groups $G(V,\omega,\gamma)$ to have an exponential map. Concerning this we have the following proposition.

\begin{prop}\label{propexpmapgen}
Let $V$ be a locally convex space, $Z$ a Mackey complete locally convex space, $\omega:V \times V \rightarrow Z$ a skew-symmetric bilinear map and $\gamma:\R\rightarrow \GL(V)$ a one-parameter group with $\omega(\gamma(t)x,\gamma(t)y)=\omega(x,y)$ for $x,y\in V,t\in \R$, and such that $\gamma$ induces a smooth action of $\R$ on $V$. Consider the Lie group $G:=(Z\times_\omega V) \rtimes_\gamma \R$ and assume that for every $x\in V$ the map $b_x(t):=\gamma(t)x$ has a $V$-valued integral, i.e., there is a smooth curve $B_x:\R\rightarrow V$ with $B_x'=b_x$. In particular $\int_a^b \gamma(t)x dt=B_x(b)-B_x(a)$ for $a,b\in \R$. Then $G$ has an exponential map given by
\begin{align}\label{expmapformulagen}
\exp(z,x,s)=\left(z+\frac{1}{2}\int_0^1\int_0^1 \omega\left(\gamma(stt')x,\gamma(st)x\right)tdt'dt\ ,\  \int_0^1\gamma(st)xdt\ ,\ s\right)
\end{align}
for $(z,x,s)\in\g=\L(G)$.
\end{prop}
\begin{proof}
Let $\exp$ be defined as in \eqref{expmapformulagen}. Differentiation under the integral sign shows that $\exp$ is smooth. Now let $(z,x,s)\in \g$ and define $\alpha(k):=\exp(kz,kx,ks)$ for $k\in \R$.
Then integration by substitution yields
\begin{align*}
\alpha(k)&=\left(kz+\frac{1}{2}\int_0^1\int_0^1 \omega\left(\gamma(sktt')x,\gamma(skt)x\right)k^2tdt'dt\ ,\  \int_0^1\gamma(skt)kxdt\ ,\ ks\right)\\
&=\left(kz+\frac{1}{2}\int_0^k\int_0^1 \omega\left(\gamma(stt')x,\gamma(st)x\right)tdt'dt\ ,\  \int_0^k\gamma(st)xdt\ ,\ ks\right).
\end{align*}
Therefore
\begin{align}\label{propexpmapgeneq1}
\alpha'(k) = \left(z+\frac{1}{2}\int_0^1\omega\left(\gamma(skt')x,\gamma(sk)x\right)kdt'\ ,\  \gamma(sk)x\ ,\ s\right).
\end{align}
For $g\in G$ denote by $l_g(h)=gh, h\in G$, the left multiplication with $g$. Then we have
\begin{align*}
d(l_{(z_1,x_1,s_1)})(0,0,0)(z_2,x_2,s_2)&=\derat0((z_1,x_1,s_1)\cdot (tz_2,tx_2,ts_2))\\
&=(z_2+\textstyle{\frac{1}{2}}\omega(x_1,\gamma(s_1)x_2),\gamma(s_1)x_2,s_2).
\end{align*}
In particular
$$d(l_{\exp(kz,kx,ks)})(0,0,0)(z,x,s)=\left(z+\frac{1}{2}\int_0^1\omega\left(\gamma(kst)kx,\gamma(ks)x\right)dt\ ,\  \gamma(ks)x\ ,\ s\right).$$
With \eqref{propexpmapgeneq1} we obtain that the logarithmic derivative of $\alpha$ satisfies $\delta(\alpha) = (z,x,s)$. Since $\exp(0,0,0)=(0,0,0)$, we conclude that $\alpha$ is a one-parameter group and therefore $\exp$ is an exponential map for $G$.
\end{proof}

\begin{rmk}
If $V$ is Mackey complete then the condition in Proposition \ref{propexpmapgen} that the maps $b_x(t)=\gamma(t)x, x\in V$, have a $V$-valued integral is automatically satisfied.
\end{rmk}

When we talk about an oscillator group $G(V,\omega,\gamma)$ we shall always assume that the maps $t\mapsto\gamma(t)x, x\in V$, have a $V$-valued integral. Hence $G(V,\omega,\gamma)$ has an exponential map and it is given by formula \eqref{expmapformulagen}.

\begin{dfn}\label{defstandosc}
Let $H\neq0$ be a complex Hilbert space, $\gamma:\R \rightarrow \U(H)$ a unitary one-parameter group and $A=A^*$ its self-adjoint generator with $A\geq0$ and $\ker A=0$. Let $V_A:=C^\infty(A)$ be the space of $\gamma$-smooth vectors in $H_A:=H$ equipped with the $C^\infty$-topology and $\omega_A(x,y):=\Im\langle Ax,y\rangle$ for $x,y\in V_A$. Then the oscillator group
$$G_A := \Heis(V_A,\omega_A)\rtimes_\gamma \R = \R\times_{\omega_A} V_A\rtimes_\gamma \R$$
is called a \textit{standard oscillator group}. Its Lie algebra is $\g_A := \heis(V_A,\omega_A) \rtimes_D \R$, where $D:=iA\vert_{V_A}$. Note $\omega_A(x,y)=-\Re\langle Dx,y\rangle$ and that $D(V_A)\subset V_A$ is dense (\cite[Lem.~4.1(b)]{NZ13}).
\end{dfn}

\begin{prop}
The adjoint action of $G_A$ is given by
\begin{equation}
\Ad(t',x',s')(t,x,s) = \left(t-\Re\la Dx',\gamma(s')x\ra+\frac{1}{2}\|Dx'\|^2s,\gamma(s')x-sDx',s\right). \label{adjointGA}
\end{equation}
For $\lambda=(t^*,\alpha,s^*)$, the coadjoint action of $G_A$ is given by
\begin{equation}\label{coadjointactionGA}
\Ad^*(t',x',s')\lambda = \left(t^*,\alpha\circ \gamma(-s') + t^* \Re \langle Dx',\cdot \rangle,\frac{t^*}{2}\|Dx'\|^2 + \alpha(D\gamma(-s')x')+s^*\right).
\end{equation}
\end{prop}
\begin{proof}
For \eqref{coadjointactionGA} see \cite[Rmk.~2.5]{NZ13} and \eqref{adjointGA} is shown similarly.
\end{proof}

Since we are mainly interested in semibounded representations of $G:=G(V,\omega,\gamma)$ we may assume by \cite[Thm.~5.9 and Rmk.~2.4]{NZ13} to have a standard oscillator group $G_A=\Heis(V_A,\omega_A) \rtimes_\gamma \R$ with Lie algebra $\g_A$ and a dense inclusion $V\hookrightarrow V_A$ such that the corresponding inclusions $G\hookrightarrow G_A$ and  $\g \hookrightarrow \g_A$ are morphisms of Lie groups resp. Lie algebras. We keep this assumption for Section \ref{sectog}. Note that the topology on $\g$ is finer than the topology on $\g_A$. 

\subsection{Open invariant cones in oscillator algebras}\label{sectionopeninvariantcones}
In this subsection we want to classify all open invariant cones in the Lie algebra $\g$. We shall see that every open invariant cone $W$ in $\g$ is completely determined by its intersection with the two-dimensional maximal abelian subalgebra $\t:=\R\times \{ 0 \} \times \R$. Furthermore we will clarify which open cones in $\t$ give rise to proper invariant open cones in $\g$.

\begin{dfn}\label{Wddef}
We define
$$W_\infty:= \R \times V_A \times]0,\infty[, \quad C_\infty:= W_\infty \cap \t = \R \times\{0\} \times ]0,\infty[$$
and
$$W_d:=\big\{(t,x,s) \in W_\infty : t-\textstyle{\frac{1}{2s}}\|x\|^2+ds>0\big\}, \quad C_d:= W_d \cap \t $$
for $d\in \R$.
\end{dfn}

\begin{rmk}
Note that $C_d=\{(t,0,s) : s>0, t+ds>0\}$. Hence the sets $C_d$ for  $d\in \overline{\R}:=\R\cup \{\infty\}$ are exactly the non-empty open cones in the upper half-plane $\R \times\{0\} \times ]0,\infty[$ for which one boundary ray is given by the positive axis $]0,\infty[ \times\{0\} \times \{0\}$.
\end{rmk}

\begin{rmk}\label{mapf_d}
For $d\in \R$ we define the map
$$f_d: W_\infty \rightarrow \R,(t,x,s)\mapsto t+ds-\textstyle{\frac{1}{2s}}\|x\|^2.$$
Note that $W_d=f_d^{-1}(\R_{>0})$ for all $d\in \R$.
\end{rmk}

\begin{thm}\label{conelemma}
Let $D(V)\subset V$ be dense. Then the following assertions hold:
\begin{itemize}
\item[\rm(a)] Every proper open invariant cone $W$ in $\g$ is either contained in $W_\infty$ or $-W_\infty$.
\item[\rm(b)] For every open invariant cone $W\subset \g \cap W_\infty$ we have $]0,\infty[ \times\{0\} \times \{0\} \subset \overline{W\cap \t}$.
\item[\rm(c)] $W_d={\rm int}(\overline{\Ad(G_A)C_d})$ is a proper open invariant cone in $\g_A$ for all $d\in \overline{\R}$.
\end{itemize}
\end{thm}
\begin{proof}
Recall the adjoint action of $G_A$ from \eqref{adjointGA}:
\begin{equation}
\Ad(t',x',s')(t,x,s) = \left(t-\Re\la Dx',\gamma(s')x\ra+\frac{1}{2}\|Dx'\|^2s,\gamma(s')x-sDx',s\right). \label{adjointGAcopy1}
\end{equation}
Since $D(V)\subset V$ is dense, we conclude
\begin{equation}
(\Ad(G)U) \cap \t \neq \emptyset \ \ \text{ for every non-empty open subset $U\subset \g$}\label{adjointGAcons0}
\end{equation}
and
\begin{equation}
\frac{1}{2}\Ad(0,x',0)(t,x,s)+\frac{1}{2}\Ad(0,-x',0)(t,x,s) = (t+\frac{1}{2}\|Ax'\|^2s,x,s). \label{adjointGAcons}
\end{equation}
(a) Suppose there is a proper open invariant cone $W$ in $\g$ which is neither contained in $W_\infty$ nor in $-W_\infty$. Then we find $(t_0,x,s), (t_0',x',s') \in W$ such that $s>0$ and $s'<0$. Since $W$ is open and invariant, we may in view of  \eqref{adjointGAcons0} achieve that $x=x'=0$. As $W$ is invariant and convex, we then conclude with \eqref{adjointGAcons} that $(t,0,s), (t',0,s')\in W$ for all $t\geq t_0$ and $t'\leq t_0'$. Let $\lambda:=-\frac{s'}{s}$. Since $W$ is a cone we have $\lambda(t,0,s)=(\lambda t,0,-s')\in W$ for all $t\geq t_0$. Now we choose $t_1\geq t_0, t_2\leq t_0'$ with $\lambda t_1+t_2=0$ and conclude that $0 = \frac{1}{2}(\lambda t_1,0,-s')+\frac{1}{2}(t_2,0,s') \in W$, a contradiction to the assumption that $W$ is proper.\newline
(b) Since $W\subset W_\infty$ equation \eqref{adjointGAcons} implies that $]0,\infty[ \times\{0\} \times \{0\} \subset \text{lim}(W\cap \t) \subset \overline{W\cap \t}$.\newline
(c) Let $d\in \R$. Obviously $W_d$ is open, proper and $\lambda W_d \subset W_d$ for $\lambda \in \R_{>0}$. Consider
$$f_d: W_\infty \rightarrow \R,(t,x,s)\mapsto t+ds-\textstyle{\frac{1}{2s}}\|x\|^2.$$
Recall $W_d=f_d^{-1}(\R_{>0})$ from Remark \ref{mapf_d}. Let $a_1=(t_1,x_1,s_1)$ and $a_2=(t_2,x_2,s_2)$ be in $W_d$. Set $y_1:=\frac{1}{s_1}x_1$ and $y_2:=\frac{1}{s_2}x_2$. By convexity of the map $x\mapsto \|x\|^2$ we have
$$\frac{s_1}{s_1+s_2}\|y_1\|^2+\frac{s_2}{s_1+s_2}\|y_2\|^2\geq \left\|\frac{s_1y_1+s_2y_2}{s_1+s_2}\right\|^2 =\frac{\|x_1+x_2\|^2}{(s_1+s_2)^2}.$$
Thus we obtain
$$2(f_d(a_1+a_2)-f_d(a_1)-f_d(a_2)) = s_1\|\textstyle{\frac{1}{s_1}x_1}\|^2+s_2\|\textstyle{\frac{1}{s_2}x_2}\|^2-\textstyle{\frac{1}{s_1+s_2}}\|x_1+x_2\|^2\geq0.$$
This shows $a_1+a_2\in W_d$ and we obtain that $W_d$ is a cone. Obviously $W_\infty$ is an open invariant cone. Now let $d\in\overline{\R}$. By \eqref{adjointGAcopy1} we obtain $\Ad(G_A)C_d\subset W_d$ and with \cite[Lem. 2.8]{Ne10b} we conclude $\text{int}(\overline{\Ad(G_A)C_d})\subset \text{int}(\overline{W_d})=W_d$. Since $A(V_A) \subset V_A$ is dense we obtain from \eqref{adjointGAcopy1} that $(\Ad(G_A)U) \cap \t \neq \emptyset$ for every non-empty open subset $U\subset \g_A$. Since $W_d$ is open we conclude that $W_d\subset \overline{\Ad(G_A)C_d}$. Hence we get $W_d = \text{int}(\overline{\Ad(G_A)C_d})$. Now we see that $W_d$ is $\Ad$-invariant.
\end{proof}

\begin{prop}\label{openinvconesgen}
Assume that $D(V)\subset V$ is dense. The non-empty proper open invariant cones in the oscillator algebra $\g$ are given by $\pm W_d \cap \g$, where $d\in \overline{\R}$. In particular, every open invariant cone in $\g$ is the intersection of an open invariant cone in $\g_A$ with $\g$.
\end{prop}
\begin{proof}
From Theorem \ref{conelemma}(c) we obtain that $\pm W_d \cap \g$ is an open invariant cone in $\g$. Now let $W \subset \g$ be a non-empty proper open invariant cone in $\g$. By Theorem \ref{conelemma}(a) we may assume that $W\subset W_\infty \cap \g$ and with Theorem \ref{conelemma}(b) we further obtain $W\cap \t = C_d$ for some $d\in \overline{\R}$. Then $\Ad(G)C_d \subset W$. From the density of $D(V)$ in $V$ and equation~\eqref{adjointGA} we further obtain $W_d\cap \g \subset \overline{W}$ (cf. definition of $W_d$). Thus $W_d\cap \g \subset \text{int}\overline{W}=W$. If $d=\infty$ we obtain $W=W_\infty\cap\g$. Hence we may assume $d\in \R$. Recall
$$f_d: W_\infty \rightarrow \R,(t,x,s)\mapsto t+ds-\textstyle{\frac{1}{2s}}\|x\|^2$$
and $W_d=f_d^{-1}(\R_{>0})$ from Remark \ref{mapf_d}. From \eqref{adjointGAcons0} we obtain $W\subset \overline{\Ad(G)C_d}\cap W_\infty\subset \overline{W_d}\cap W_\infty$. Thus $f_d(a)\geq0$ for all $a\in W$ since $\overline{W_d}\cap W_\infty$ is the closure of $W_d$ in $W_\infty$ because $W_\infty$ is open in $\g_A$. Since $W$ is open in $\g$ the equality $f_d(a_0)=0$ for some $a_0\in W$ would imply the existence of a $a_1 \in W$ with $f_d(a_1)<0$ which is a contradiction. Thus $W\subset f_d^{-1}(\R_{>0})=W_d$. Hence in all, we get $W=W_d\cap \g$. This completes the proof.
\end{proof}

\begin{rmk}\label{openinvconesgenrmk}
The cones $W_d,d\in \overline{\R}$ are open in $\g_A=\R\times V_A\times\R$ when $V_A$ is equipped with the $\|\cdot\|$-topology, since $W_d=f_d^{-1}(\R_{>0})$ and $f_d: W_\infty \rightarrow \R$ is $\|\cdot\|$-continuous for $d\in \R$. Hence if $D(V)\subset V$ is dense, then by the preceding proposition the open invariant cones in $\g=\R\times V\times\R$ are also open when $V$ is equipped with the $\|\cdot\|$-topology.
\end{rmk}

\subsection{Semibounded representations}\label{settingsbrepsect}
Let $\pi:G=G(V,\omega,\gamma)\rightarrow \U(\H)$ be a semibounded representation. If $D(V)\subset V$ is dense, then the open invariant cone $B(I_\pi)^0$ is either equal to $\g$ or contained in $W_\infty$ or $-W_\infty$ (Theorem~\ref{conelemma}(a)).

\begin{lem}\label{centraltrivrep}
Let $\pi:G \rightarrow \U(\H)$ be a smooth representation with $B(I_\pi)^0\neq\emptyset$ and such that $I_\pi \subset \{0\} \times V' \times \R$. Then $\pi\vert_{\R\times D(V)\times\{0\}}$ is trivial.
\end{lem}
\begin{proof}
Let $\lambda=(0,\alpha,s^*)\in I_\pi$ and choose $(t,x,s)\in B(I_\pi)^0,s\neq0$. By \eqref{coadjointactionGA} we have
$$(\Ad^*(0,v,0)\lambda)(t,x,s) = \alpha(x)+s^*s+s\alpha(Dv)$$
for $v\in V$. Since $\inf_{v\in V}(\Ad^*(0,v,0)\lambda)(t,x,s) > -\infty$ we conclude that $\alpha(D(V))=\{0\}$. Thus we obtain $I_\pi \subset (\R\times D(V)\times \{0\})^\perp$, i.e., $\big\langle\dd\pi(a)w,w\big\rangle=0$
for all $a\in [\g,\g]=\R\times D(V)\times \{0\}$ and $w\in \H^\infty$. Hence $\R\times D(V)\times \{0\} \subset \ker(\dd\pi)$ since $\H^\infty\subset \H$ is dense. Now the statement follows.
\end{proof}

\begin{prop}[Bounded representations]\label{boundedreps}
Let $\pi:G \rightarrow \U(\H)$ be a smooth representation with $B(I_\pi)=\g$. Then $\pi\vert_{\R\times D(V)\times\{0\}}$ is trivial. In particular, if $D(V)\subset V$ is dense, then $\pi\vert_{\Heis(V,\omega)}$ is trivial.
\end{prop}
\begin{proof}
Recall our assumption $V\neq0$. Let $\lambda=(t^*,\alpha,s^*)\in I_\pi$ and choose $0\neq x \in V$ and $v\in V$ with $\omega(v,x)\neq0$. By equation \eqref{coadjointactionGA}
$$(\Ad^*(0,cv,0)\lambda)(0,x,0) = \alpha(x)-ct^*\omega(v,x)$$
for all $c\in\R$. Since $B(I_\pi)=\g$ we have $(0,x,0)\in B(\cO_\lambda)$ and thus 
$$-\infty < \inf_{c\in \R}(\Ad^*(0,cv,0)\lambda)(0,x,0) = \alpha(x)-\sup\{ct^*\omega(v,x) : c\in \R\}.$$
Hence $t^*=0$ since $\omega(v,x)\neq0$. Thus we obtain $I_\pi\subset \{0\} \times V' \times \R$. Lemma \ref{centraltrivrep} now implies that $\pi\vert_{\R\times D(V)\times\{0\}}$ is trivial.
\end{proof}

\begin{rmk} \label{rem:switchomega} Note that the map 
$$ \psi : \heis(V,\omega) \rtimes_D \R \to \heis(V,-\omega) \rtimes_{D} \R, 
\quad \psi(t,v,s) := (-t,v,s)$$ 
defines an isomorphism of Lie algebras. An analogous 
statement holds for the corresponding Lie groups.
\end{rmk}

\begin{rmk}\label{settingsbrep}
Let $\pi:G=G(V,\omega,\gamma)\rightarrow \U(\H)$ be a semibounded representation and assume $D(V)\subset V$ dense. By applying \cite[Rmk.~2.4]{NZ13} and Remark \ref{rem:switchomega} we may switch from $B(I_\pi)^0$ to $-B(I_\pi)^0$ while $\omega(Dx,y)$ remains positive definite, so that we still have $G\subset G_A$ as above. Thus we may further assume that $B(I_\pi)^0\subset W_\infty$ (Proposition~\ref{boundedreps} and Theorem~\ref{conelemma}(a)).
\end{rmk}

\begin{lem}\label{adinvariantmapF}
The map 
$$F:\{ (t,x,s)\in\g : s\neq0\} \rightarrow \t, (t,x,s)\mapsto \big(t-\textstyle{\frac{\|x\|^2}{2s}},0,s\big)$$
is $\|\cdot\|$-continuous, open and $\Ad$-invariant. Moreover, if $D(V)\subset V$ is dense, then
$$(t,x,s)\in\overline{\Ad(G)F(t,x,s)} \quad\text{and}  \quad F(t,x,s)\in\overline{\Ad(G)(t,x,s)}$$
hold for all $(t,x,s)\in\g, s\neq0$.
\end{lem}
\begin{proof}
Recall the adjoint action \eqref{adjointGA}:
$$\Ad(t',x',0)(t,x,s) = \big(t-\Re\langle Dx',x\rangle+\textstyle{\frac{s}{2}}\|Dx'\|^2,x-sDx',s\big).$$
Hence
\begin{align*}
F(\Ad(t',x',0)(t,x,s)) &=\big(t-\Re\langle x,Dx'\rangle+\textstyle{\frac{s}{2}}\|Dx'\|^2-\textstyle{\frac{\|x-sDx'\|^2}{2s}},0,s\big) \\
&= \big(t-\textstyle{\frac{\|x\|^2}{2s}},0,s\big) = F(t,x,s).
\end{align*}
Moreover $F(\Ad(0,0,s')(t,x,s))=F(t,\gamma(s')x,s)=F(t,x,s)$. Hence $F$ is $\Ad$-invariant. Obviously $F$ is $\|\cdot\|$-continuous and open. Now suppose that $D(V)\subset V$ is dense and let $(t,x,s) \in \g, s\neq0$. Then
\begin{equation*}
\Ad\Big(0,-\textstyle{\frac{x'}{s}},0\Big)\Big(t-\textstyle{\frac{\|x\|^2}{2s}},0,s\Big) = \Big(t-\textstyle{\frac{\|x\|^2}{2s}}+\textstyle{\frac{\|Dx'\|^2}{2s}},Dx',s\Big).
\end{equation*}
shows that $(t,x,s)\in\overline{\Ad(G)F(t,x,s)}$. That $F(t,x,s)$ lies in $\overline{\Ad(G)(t,x,s)}$ is shown by a similar argument.
\end{proof}
\begin{rmk}
Note that $F$ restricts to the identity on $\{(t,0,s)\in\t : s\neq0\}$.  Thus, if the adjoint orbit $\cO_a$ of $a:=(t,x,s)\in\g, s\neq0$ intersects $\t$, the element $F(t,x,s)$ is the unique intersection point of $\cO_a$ with $\t$. If $D:V\rightarrow V$ is not surjective then there are orbits $\cO_a$ which do not intersect $\t$.
\end{rmk}

\begin{prop}\label{supportfunctiongen}
Let $\pi:G\rightarrow \U(\H)$ be a smooth representation with $G\subset G_A$ as above. Assume $D(V)\subset V$ dense. Then the support function satisfies
\begin{equation}\label{supfuncformula}
s_\pi(t,x,s)=s_\pi\Big(t-\frac{\|x\|^2}{2s},0,s\Big)
\end{equation}
for all $(t,x,s) \in \g$ with $s\neq0$. Moreover $s_\pi\vert_{B(I_\pi)^0}$ is continuous when $V$ is equipped with the $\|\cdot\|$-topology.
\end{prop}
\begin{proof}
Recall $F$ from Lemma \ref{adinvariantmapF} and let $(t,x,s) \in \g, s\neq0$. As a supremum of linear functionals, the support function $s_\pi:\g \rightarrow \R\cup \{\infty\}$ is lower semicontinuous, i.e. the sets $\{a\in \g : s_{\pi}(a)\leq c\}$ are closed in $\g$ for all $c\in\R$. Since $s_\pi$ is also $\Ad$-invariant, we obtain from $(t,x,s)\in\overline{\Ad(G)F(t,x,s)}$ (see Lemma \ref{adinvariantmapF}) that $(t,x,s)$ lies in the closed set 
$$\{a\in \g : s_{\pi}(a)\leq s_\pi(F(t,x,s))\}.$$
Hence $s_\pi(t,x,s)\leq s_\pi(F(t,x,s))$. Moreover $F(t,x,s)\in\overline{\Ad(G)(t,x,s)}$ by Lemma~\ref{adinvariantmapF}, which implies $s_\pi(F(t,x,s))\leq s_\pi((t,x,s))$ in the same fashion. Therefore we have shown that $s_\pi(t,x,s)= s_\pi(F(t,x,s))=s_\pi\Big(t-\frac{\|x\|^2}{2s},0,s\Big)$.

If $B(I_\pi)^0=\g$ then Proposition \ref{boundedreps} implies that $s_\pi(t,x,s)=s\cdot s_\pi(0,0,1)$ for $s\geq 0$ and $s_\pi(t,x,s)=-s\cdot s_\pi(0,0,-1)$ for $s\leq 0$. In particular, then $s_{\pi}$ is continuous. In view of Proposition \ref{openinvconesgen} we thus may assume $B(I_\pi)^0 \subset W_\infty$, i.e. we have for every $(t,x,s) \in B(I_\pi)^0$ that $s>0$. From \eqref{supfuncformula} we obtain $F(B(I_\pi))\subset B(I_\pi)\cap\t$ and since $F$ is open, $F(t,x,s)=\Big(t-\frac{\|x\|^2}{2s},0,s\Big)$ lies in the interior of $B(I_\pi) \cap \t$ in $\t$ for every $(t,x,s)\in B(I_\pi)^0$. As $\t$ is finite-dimensional the convex function $s_\pi\vert_{\text{int}(B(I_\pi) \cap \t)}$ is continuous. Thus we conclude with \eqref{supfuncformula} that $s_\pi\vert_{B(I_\pi)^0}$ is $\|\cdot\|$-continuous.
\end{proof}

\begin{rmk}\label{supportfunctiongenrmk}
Suppose we are in the situation of Proposition \ref{supportfunctiongen} and let $W_d\cap \g\subset B(I_\pi)^0$ for some $d\in \overline{\R}$. The Definition \ref{Wddef} of $W_d$ shows that $(t-\frac{\|x\|^2}{2s},0,s)\in W_d\cap \t$ holds for all $(t,x,s) \in W_d$. Since $W_d\cap\t\subset B(I_\pi)^0\cap\t$ the map
$$\hat s:W_d \rightarrow \R, (t,x,s) \mapsto s_\pi\Big(t-\frac{\|x\|^2}{2s},0,s\Big)$$
is a $\|\cdot\|$-continuous extension of $s_\pi\vert_{W_d\cap \g}$ to the open subset $W_d\subset \g_A$.
\end{rmk}

\begin{cor}
Let $D(V)\subset V$ be dense. A smooth representation $\pi:G\rightarrow \U(\H)$ is semibounded if and only if the set $B(I_\pi)$ contains interior points.
\end{cor}
\begin{proof}
If $B(I_\pi)^0\neq\emptyset$ then $s_\pi$ is locally bounded on $B(I_\pi)^0$ by Proposition \ref{supportfunctiongen}, so that $\pi$ is semibounded.
\end{proof}

\begin{lem}\label{lemma0}
Let $\pi:G_A \rightarrow \U(\H )$ be a smooth representation. If $-i\dd\pi(0,0,1)$ is bounded from below then $\pi$ is semibounded. If $\pi$ is semibounded with $\pi(t,0,0)=e^{it}\1$ then $-i\dd\pi(0,0,1)$ is bounded from below and $B(I_{\pi})^0 = \R \times V_A \times ]0, \infty[$.
\end{lem}
\begin{proof}
Let $b:=s_\pi(0,0,1)<\infty$. By passing to $\pi\otimes \chi_a$ for a suitable character $\chi_a(t,x,s)=e^{ias},a\in\R$, on $G_A$ we may assume that $b=0$. Consider the closed cone $C:=s_\pi^{-1}(]-\infty,0])$ in $\g_A$. By Proposition~\ref{supportfunctiongen} the paraboloid $\big\{\big(\frac{\|x\|^2}{2},x,1\big):x\in V_A\big\}$ is contained in $C$ hence also its closed convex hull. As $C$ is a cone we conclude $C^0 \neq \emptyset$ and thus $\pi$ is semibounded. Now assume $\pi(t,0,0)=e^{it}\1$. Then $I_{\pi}\subset \{1\} \times  V_A'  \times \R$. Let $\lambda=(1,\alpha,s^*)\in I_\pi$. By Proposition \ref{coadjointactionGA}
$$(\Ad^*(0,v,0)\lambda)(t,x,s)=t+\alpha(x)+ \Re\langle Dv,x\rangle+ \textstyle{\frac{s}{2}}\|Dv\|^2+s\alpha(Dv)+s^*s,$$
which is a quadratic form in $v$. If $(t,x,s)\in B(I_\pi)^0$ it is bounded from below and thus $s>0$. Hence $B(I_\pi)^0\subset \R \times V_A \times ]0, \infty[$. By Proposition \ref{supportfunctiongen} and $\pi(t,0,0)=e^{it}\1$
\begin{equation}\label{suppfuncspec}
s_\pi(t,x,s)=\frac{\|x\|^2}{2s}-t+s\cdot s_\pi(0,0,1)
\end{equation}
for $s\neq0$. If $\pi$ is semibounded we may choose $(t_0,x_0,s_0)\in B(I_\pi)^0$. Then \eqref{suppfuncspec} implies $(0,0,1)\in B(I_\pi)$ and furthermore $\R\times V_A \times ]0, \infty[ \ \subset B(I_\pi)^0$.
\end{proof}

\begin{thm}
Let $G\subset G_A$ be as above and assume that $D(V)\subset V$ is dense. Then every semibounded representation $\pi:G\to\U(\H)$ extends (uniquely) to a semibounded representation $\pi^\sharp:G_A\to\U(\H)$.
\end{thm}
\begin{proof}
By  Remark~\ref{settingsbrep} we may assume that $B(I_\pi)^0 \subset W_\infty$. Hence $B(I_\pi)^0=W_d\cap \g$ for some $d\in \overline{\R}$ (Proposition~\ref{openinvconesgen}). By Remark \ref{supportfunctiongenrmk} the support functional $s_\pi\vert_{B(I_\pi)^0}$ extends to a continuous map $\hat s:W_d \rightarrow \R$. Thus the assertion follows from \cite[Prop.~6.2]{NSZ15}.
\end{proof}

\section{Complex semigroups associated to $G_A$}\label{sectcomplexsemigroupsassoctoG_A}

In this section we associate complex involutive semigroups to the standard oscillator group $G_A$. More precisely, we consider a dense subgroup $G_A^\cO\subset G_A$ (with a finer topology) which admits a complexification $G_{A,\C}$. For the open invariant cones $W_d\subset \g_A, d\in\overline\R$, we then consider the subset $S_d=G_A^\cO\cdot\exp_\C(iW_d^\cO)$ of $G_{A,\C}$, where $W_d^\cO=W_d\cap \g_A^\cO$. We show that $S_d$ is an open complex subsemigroup which admits a natural involution turning it into a complex involutive semigroup. The semigroups $S_d$ will be useful for studying the semibounded representations of $G_A$ (cf. Section \ref{applsbrepr}).

\subsection{Complex oscillator groups}

Let us first recall the definition of a complex involutive semigroup from \cite{Ne08}.

\begin{dfn}\label{complexsemigrpdef}
\begin{itemize}
\item[\rm(a)] An \textit{involutive complex semigroup} is a complex manifold $S$ modeled on a locally convex space endowed with a holomorphic semigroup multiplication and an involution denoted by $*:S\rightarrow S,s\mapsto s^*$ which is an antiholomorphic antiautomorphism.
\item[\rm(b)] A function $\alpha: S\rightarrow \R_{\geq0}$ is called an \textit{absolute value} if
$$\alpha(s)=\alpha(s^*) \quad \text{ and } \quad \alpha(st) \leq \alpha(s)\alpha(t)$$
for all $s,t\in S$.
\item[\rm(c)] A \textit{holomorphic representation} $\pi$ of a complex involutive semigroup $S$ on the Hilbert space $\H$ is a morphism $\pi:S\rightarrow B(\H)$ of semigroups with $\pi(s^*)=\pi(s)^*$ for $s\in S$ such that $\pi$ is holomorphic if $B(\H)$ is endowed with its natural complex Banach space structure defined by the operator norm. If $\alpha$ is an absolute value on $S$, then the representation $\pi$ is said to be \textit{$\alpha$-bounded} if $\|\pi(s)\| \leq \alpha(s)$ holds for each $s\in S$. The representation $\pi$ is called \textit{non-degenerate} if  $\pi(S)v = \{0\}$ implies $v=0$.
\end{itemize}
\end{dfn}

Consider a standard oscillator group $G_A=\Heis(V,\omega) \rtimes_{\gamma} \R$, where $V:=C^{\infty}(A)\subset H_A$. Denote the complex structure on $V$ resp. $H_A$ by $I$. Consider the real inner product $b(x,y):=\Re\langle x,y\rangle$ on $H_A$ and let $((H_A)_\C,\langle \cdot, \cdot \rangle_\C)$ be the complex inner product space obtained by complexification of $(H_A,b)$, where $(H_A)_\C=H_A\oplus iH_A$. By complex-linear extension we obtain $\gamma:\R \rightarrow \U((H_A)_\C)$ and denote its self-adjoint generator by $A_c$. Thus $iA_c\vert_{V_\C}=D_\C$ where $D_\C:V_\C\rightarrow V_\C$ is the complex-linear extension of $D$. Consider the canonical conjugation $\sigma$ on $(H_A)_\C, \sigma(x+iy)=x-iy$. Recall from Proposition \ref{equicontholvecprop} the dense subspace $V_\C^\cO \subset V_\C$, which is a Fr\'echet space when equipped with the $C^\cO$-topology, and the action
$$\gamma_{\C}:\C \times V^{\cO}_{\C} \rightarrow V^{\cO}_{\C},\gamma_{\C}(z,v)=\gamma_{\C}(z)(v)=\sum_{k\geq0}\frac{1}{k!}z^kD_\C^kv$$
which extends the action $\gamma$ on $V^{\cO}$. Note that $(V_{\C})^{\cO}=(V^{\cO})_{\C}$ by Remark \ref{equicontcomplexificationrmk}. Let $\omega_\C: V^\cO_\C \times V^\cO_\C \rightarrow \C$ be the complex bilinear extension of $\omega_A\vert_{V^\cO \times V^\cO}$.
\begin{dfn}\label{complexsemigroupsG_A}
\begin{itemize}
\item[\rm(a)] Define the complex Lie group
$$G_{A,\C}:=\Heis_\C(V_{\C}^{\cO},\omega_\C) \rtimes_{\gamma_{\C}} \C$$
with Lie algebra $\g_{A,\C}^{\cO}$. Since $V^{\cO}_\C$ is a Fr\'echet space, $G_{A,\C}$ is a Fr\'echet--Lie group.
\item[\rm(b)] We also consider the Fr\'echet--Lie group 
$$G_A^{\cO}:=\Heis(V^\cO,\omega) \rtimes_{\gamma} \R$$
with Lie algebra $\g_A^\cO$. Obviously $\g_{A,\C}^{\cO}$ is the complexification of $\g_A^\cO$. Furthermore $G_A^{\cO}\subset G_A$ is a dense subgroup of $G_A$ but the topology on $G_A^{\cO}$ is finer than that on~$G_A$.
\item[\rm(c)] Define $S_A:=\{(z,x,s) \in G_{A,\C} : \Im s > 0 \}$, which is an open complex subsemigroup  in $G_{A,\C}$. We set $W_\infty^\cO:=\R \times V^{\cO} \times ]0,\infty[$.
\end{itemize}
\end{dfn}

The complex Lie group $G_{A,\C}$ has an exponential map which is given by
\begin{align}\label{expformula}
\exp_{\C}(z,x,s)=\left(z+\frac{1}{2}\int_0^1\int_0^1 \omega_{\C}\left(\gamma_{\C}(stt')x,\gamma_{\C}(st)x\right)tdt'dt\ ,\  \int_0^1\gamma_{\C}(st)xdt\ ,\ s\right)
\end{align}
with $z,s \in \C$ and $x\in V^{\cO}_{\C}$ (Proposition \ref{propexpmapgen}). Certainly $\exp_\C\vert_{\g_A^\cO}$ is the exponential map of $G_A^\cO$. Let us give another expression for $\exp_\C$. We define by functional calculus of the self-adjoint operator $A_c$ on $(H_A)_\C$ the operators
$$B_z:=\frac{e^{iA_c z}-1}{iA_c z}$$
for $z\in \C^\times$ and set $B_0:=\id \in \End((H_A)_\C)$. Note that $V_\C^\cO$ lies in the domain of $B_z$ because the vectors in $V_\C^\cO$ are holomorphic for the unitary one-parameter group generated by $A_c$. Obviously $B_z\vert_{V_\C^\cO}=(D_\C z)^{-1}(\gamma_\C(z)-1)$, in particular $V_\C^\cO$ is invariant under $B_z$ .

\begin{prop}
For all $(z,x,s)\in \g_{A,\C}$ with $s\neq0$ we have:
\begin{align}\label{expformulaspec}
\exp_\C(z,x,s)= \left(z+\frac{1}{2s}\left\langle x,\sigma(B_sx-x) \right\rangle_\C \ ,\ B_sx\ ,s\right).
\end{align}
\end{prop}
\begin{proof}
Obviously $B_sx=\int_0^1\gamma_{\C}(st)xdt$ for $s\neq0$. Note that $\omega_\C(x,y)=-\langle D_\C x,\sigma y \rangle_\C$ and
$$\langle \gamma_\C(z)x,\sigma\gamma_\C(z) x \rangle_\C = \langle \gamma_\C(z)x,\gamma_\C(\overline{z}) \sigma x\rangle_\C = \langle \gamma_\C(-z)\gamma_\C(z)x,\sigma x \rangle_\C = \langle x,\sigma x\rangle_\C$$
for $x,y\in V_\C^\cO, z\in \C$.  Hence
\begin{align*}
s\cdot\int_0^1\int_0^1 \omega_{\C}\left(\gamma_{\C}(stt')x,\gamma_{\C}(st)x\right)tdt'dt &= - \int_0^1 \langle \gamma_\C(st)x-x,\sigma \gamma_\C(st)x \rangle_\C \ dt \\
&= \int_0^1 \langle x,\sigma\gamma_\C(st)x \rangle_\C \ dt - \langle x,\sigma x\rangle_\C = \langle x,\sigma(B_sx-x)\rangle_\C
\end{align*}
for all $x\in V_\C^\cO, s\in \C^\times$. In view of \eqref{expformula} we now obtain the statement.
\end{proof}

Note also that $D(V^\cO)\subset V^\cO$ is dense. This follows from Proposition \ref{Otopprop} since $D(V^b)\subset V$ is norm dense, which again follows from the norm density of both $D(V)$ and $V^b$ in $V$.

\begin{prop}\label{involutiononGAC}
Let $\sigma: V_\C^\cO\rightarrow V_\C^\cO$ denote the canonical complex conjugation on $V_\C^\cO$. The map 
$$*: G_{A,\C} \rightarrow G_{A,\C}, (z,v,s)\mapsto (\overline{z},\sigma(v),\overline{s})^{-1}$$
is an involutive antiholomorphic antiautomorphism. We have
$$(g\exp_\C(iw))^*=\exp_\C(iw)g^{-1}$$
for all $g\in G_A^\cO, w \in \g_A^\cO$.
\end{prop}
\begin{proof}
Let $\tau: G_{A,\C} \rightarrow G_{A,\C},(z,v,s)\mapsto (\overline{z},\sigma(v),\overline{s})$. Obviously $\tau$ is involutive and antiholomorphic. Since $\omega_\C$ is the complex bilinear extension of $\omega_A\vert_{V^\cO\times V^\cO}$, we have $\overline{\omega_\C(v,w)} = \omega_\C(\sigma(v),\sigma(w))$ for all $v,w\in V_\C^\cO$. Moreover $\sigma(\gamma_\C(z)v) = \gamma_\C(\overline{z})\sigma(v)$ for $z\in \C,v\in V_\C^\cO$. Thus we conclude with the explicit formula for the multiplication in $G_{A,\C}$ that $\tau$ is a Lie group homomorphism. Since the inversion $h\mapsto h^{-1}$ in the complex Lie group $G_{A,\C}$ is a holomorphic antiautomorphism, $*$ is an antiholomorphic antiautomorphism. Obviously $\tau\vert_{G_A^\cO} = \id_{G_A^\cO}$. Note that $\L(\tau)(it,ix,is)=(-it,-ix,-is)$ for $(t,x,s)\in \g_A^\cO$ and therefore $\tau(\exp_\C(iw))=\exp_\C(-iw)$ for $w \in \g_A^\cO$. We obtain 
\[(g\exp_\C(iw))^*= \tau(\exp_\C(iw))^{-1}\tau(g)^{-1}= \exp_\C(iw)g^{-1}.\qedhere\]
\end{proof}

\subsection{Polar Decomposition for $G_{A,\C}$ and $S_A$}\label{sectpolardecG_A}

In this subsection we obtain a polar decomposition of $G_{A,\C}$ and $S_A$. More specifically we want to show that the map 
$$\psi: G_A^\cO \times \g_A^\cO \rightarrow G_{A,\C}, (g,w)\mapsto g\exp_\C(iw)$$
is a diffeomorphism. Note that we cannot use the Inverse Function Theorem here, since the involved Lie groups are Fr\'echet--Lie groups and in general not Banach. Moreover we want to consider smoothness resp. analyticity properties of $\psi^{-1}$ and $\psi^{-1}\vert_{S_A}$ w.r.t. certain topologies. To do so we shall explicitly determine $\psi^{-1}$. Note that, for $s=g\exp_\C(iw)$, we have $s^*s = \exp_\C(2iw)$. Thus the main part to determine $\psi^{-1}$ is to determine the inverse of $\g_A^\cO  \ni w \mapsto \exp_\C(iw)$. In view of \eqref{expformulaspec}, we should therefore express $x\in V^\cO$ in terms of $B_{is}ix\in V^\cO_\C, s\in \R$.  To obtain a suitable form of $\psi^{-1}$, which allows us to show smoothness properties w.r.t. certain topologies, we shall consider the real and imaginary part of $B_{is}ix$ separately. As we shall see in a moment (see equation \eqref{polardecomplemeq1}), this will lead us to the following functions:

\begin{rmk}
For $z\in \C\backslash \pi i\Z$ define 
$$\tilde{f}(z):=\frac{2z}{e^{z}-e^{-z}} \quad \text{ and } \quad \tilde{g}(z):= \frac{e^{z}+e^{-z}-2}{i(e^{z}-e^{-z})}.$$
The singularities of $\tilde{f}$ and $\tilde{g}$ at $z=0$ are removable by defining $\tilde{f}(0)=1$ and $\tilde{g}(0)=0$. Thus the maps $\tilde{f}, \tilde{g}:\{0\} \cup \C \backslash \pi i\Z \rightarrow \C$ are holomorphic. Let 
$$\C_r:=\{z\in \C : \Re z \neq 0\} = \C \backslash i\R.$$
We define 
$$f:(\C_r\cup \{0\}) \times \R \rightarrow \C, (s,x) \mapsto \tilde{f}(xs) \quad\text{ and } \quad g:(\C_r\cup \{0\}) \times \R \rightarrow \C, (s,x) \mapsto \tilde{g}(xs).$$
Then $f(\cdot,x)$ and $g(\cdot,x)$ are holomorphic on $\C_r$ for each $x\in \R$. Note that for every $s\in \C_r$ we have $\lim_{x\rightarrow \pm\infty} f(s,x) = 0$ and $\lim_{x\rightarrow \pm\infty} g(s,x)\in \{-i,i\}$. In particular $f(s,\cdot)$ and $g(s,\cdot)$ are bounded smooth functions for each $s\in \C_r$. Moreover the maps $f\vert_{\R\times\R}$ and $g\vert_{\R\times\R}$ are smooth. We set $f_s(x):=f(s,x)$ and $g_s(x):=f(s,x)$ for $(s,x)\in (\C_r\cup\{0\})\times \R$. 
\end{rmk}

\begin{dfn}\label{deff(A)g(A)}
By functional calculus of $A$ we define the operators $f_s(A),g_s(A)$, i.e.
$$f_s(A)= \frac{2sA}{e^{sA}-e^{-sA}} \in B(H_A), \quad g_s(A)= \frac{e^{sA}+e^{-sA}-2}{I(e^{sA}-e^{-sA})} \in B(H_A)$$
for $s\in \C_r$ and $f_0(A)=\id, g_0(A)=0$. Note that $f_s(A),g_s(A) \in B(H_A)^\gamma$, where $B(H_A)^\gamma$ denotes the space of bounded operators on $H_A$ which commute with $\gamma(t)$ for every $t\in \R$. Hence the operators $f_s(A)$ and $g_s(A)$ leave the subspace $V^\cO$ invariant by Lemma \ref{equiconthollem1}(a). In particular we may also consider them as operators on $V^\cO$.
\end{dfn}

Recall the operators $B_z=\frac{e^{iA_c z}-1}{iA_c z}$ on $(H_A)_\C$ for $z\in \C^\times$ and $B_0:=\id \in \End((H_A)_\C)$.

\begin{lem}
For all $s\in \R^\times, x\in V^\cO$ we have
\begin{align}\label{polardecomplemeq1}
B_{is}ix=g_s(A)f_s(A)^{-1}x+if_s(A)^{-1}x.
\end{align}
\end{lem}
\begin{proof}
We calculate:
\begin{align*}
B_{is}x &= (iD_\C s )^{-1}\left(-x+\sum_{k=0}^\infty\frac{1}{k!}(is)^k D^kx\right) \\
&= (iD_\C s)^{-1}\left( -x + \sum_{k=0}^\infty\frac{1}{(2k)!}(sI)^{2k} D^{2k}x + iI^{-1}\sum_{k=0}^\infty\frac{1}{(2k+1)!}(sI)^{2k+1} D^{2k+1}x \right)\\
&= (2iD_\C s)^{-1}\left( -2x + e^{sA}x+e^{-sA}x -iI^{-1}(e^{sA}x-e^{-sA}x) \right) \\
&= \frac{1}{i}g_s(A)f_s(A)^{-1}x+f_s(A)^{-1}x.
\end{align*}
Now the assertion follows by multiplying the last equation with $i$.
\end{proof}

The purpose of the next four lemmas is to obtain smoothness and analyticity properties of $f_s(A)$ and $g_s(A)$ w.r.t. certain topologies.

\begin{lem}\label{ftsechsq}
For all $x\in \R$ we have 
$$\tilde{f}(x) = \int_\R \hat f(t)e^{ixt} dt, \ \text{ where }\ \hat f(t):= \frac{\pi}{(e^{\frac{\pi t}{2}}+e^{-\frac{\pi t}{2}})^2}.$$
\end{lem}
\begin{proof}
This can be shown with the Residue Theorem and for a proof we refer to \cite[Sect. A.15 Eq. (A.15.6)]{Ha00}.
\end{proof}

\begin{dfn}
Let $U_1\subset X_1$ and $U_2\subset X_2$ be open subsets of the topological spaces $X_1$ resp. $X_2$ and $f:U_1 \times U_2 \rightarrow \C, (x_1,x_2)\mapsto f(x_1,x_2)$ be a map. We say that $f(x_1,\cdot)$ is a \textit{bounded function locally uniformly in $x_1$}, if for every $p\in X_1$ there is a neighborhood $N$ of $p$ in $X_1$ such that $\sup_{x_1\in N,x_2\in U_2} f(x_1,x_2) <\infty$. In this case we also say that $f(x_1,x_2)$ is a \textit{bounded function in $x_2$ locally uniformly in $x_1$}.
\end{dfn}

\begin{lem}\label{polardecomplemma0}
The maps 
$$b_1:\C_r\rightarrow B(H_A)^\gamma,s\mapsto f_s(A) \quad \text{and}\quad b_2:\C_r\rightarrow B(H_A)^\gamma,s\mapsto g_s(A)$$
are holomorphic.
\end{lem}
\begin{proof}
For $s\in \C_r$ we calculate
$$\frac{\partial f}{\partial s}(s,x) = \frac{2x}{e^{xs}-e^{-xs}}- \frac{2sx^2(e^{xs}+e^{-xs})}{(e^{xs}-e^{-xs})^2}$$
if $x\neq 0$ and $\frac{\partial f}{\partial s}(s,0)=0$. From this we see that $\frac{\partial f}{\partial s}(s,\cdot)$ is a bounded function locally uniformly in $s\in \C_r$. Note that 
$$g(s,x)=\frac{\left(e^{\frac{xs}{2}}-e^{-\frac{xs}{2}}\right)^2}{i(e^{\frac{xs}{2}}+e^{-\frac{xs}{2}})(e^{\frac{xs}{2}}-e^{-\frac{xs}{2}})}=\frac{e^{xs}-1}{i(e^{xs}+1)}$$
and therefore
$$\frac{\partial g}{\partial s}(s,x) = \frac{(e^{xs}+1)e^{xs}x-(e^{xs}-1)e^{xs}x}{i(e^{xs}+1)^2}=\frac{2x}{i(e^{-xs}+e^{xs}+2)}$$
for all $s\in \C_r, x\in \R$. Hence $\frac{\partial g}{\partial s}(s,\cdot)$ is a bounded function locally uniformly in $s\in \C_r$. With Lemma \ref{equiconthollem2} we conclude that the maps $b_1$ and $b_2$ are holomorphic. 
\end{proof}

\begin{lem}\label{polardecomplemma1}
\begin{itemize}
\item[\rm(a)] The maps
$$\C_r \times V^\cO \rightarrow V^\cO, (s,v)\mapsto f_s(A)v \quad \text{ and } \quad \C_r \times V^\cO \rightarrow V^\cO, (s,v)\mapsto g_s(A)v$$
are both holomorphic. Moreover, for every $k\in \N_0$, these maps are still holomorphic when $V^\cO$ is equipped with the topology given by the norm $x\mapsto \|A^k x\|$. Furthermore these maps are holomorphic when $V^\cO$ is equipped with the $C^\infty$-topology.
\item[\rm(b)] The map 
$$h:\R \times \C_r \times V^\cO \rightarrow V^\cO, (t,s,v)\mapsto \gamma(t)f_s(A)v$$
is analytic when $V^\cO$ is equipped with the $C^\infty$-topology as well as when $V^\cO$ is equipped with the $\|\cdot\|$-topology.
\end{itemize}
\end{lem}
\begin{proof}
(a) is an immediate consequence of Lemma \ref{polardecomplemma0} and Lemma \ref{equiconthollem1}(b).\newline
(b) Let $(t_0,s_0)\in \R\times \C_r$. Then we find small neighborhoods $U$ of $t_0$ in $\C$ and $U'$ of $s_0$ in $\C_r$ such that the operators
$$h_{z,s}(A):=\frac{2sAe^{zIA}}{e^{sA}-e^{-sA}}$$
defined by functional calculus are uniformly bounded for $z\in U, s\in U'$. Then the map
$$T: U\times U' \rightarrow B(H_A)^\gamma, (z,s)\mapsto h_{z,s}(A)$$
is locally bounded. Let $v\in V^\cO$. By part (a) the map 
$$U'\rightarrow H_A,s\mapsto T(z,s)v=f_s(A)e^{zIA}v$$
is holomorphic for every $z\in U$ since $e^{zIA}v\in V^\cO$. By Proposition \ref{equicontholvecprop}(c) the map \break $U\rightarrow H_A,z\mapsto T(z,s)v=e^{zIA}f_s(A)v$ is holomorphic for every $s\in U'$. With Hartog's Theorem we conclude that $(z,s)\mapsto T(z,s)v$ is holomorphic for every $v\in V^\cO$. Since $V^\cO\subset H_A$ is dense and $T$ is locally bounded this yields that $T$ is holomorphic, cf. \cite[Lem. 3.4]{Ne08}. Now we obtain the assertion from Lemma \ref{equiconthollem1}(b), since the map $U\times U'\times V^\cO \rightarrow V^\cO, (z,s,v)\mapsto T(z,s)v$ is a local extension of $h$.
\end{proof}

\begin{lem}\label{polardecomplemma1.5}
The map $\alpha: \R \times V^\cO \rightarrow V^\cO, (s,v) \mapsto f_s(A)v$ is smooth.
\end{lem}
\begin{proof}
Recall $f_s(x)=\tilde f(sx)$ for $s,x\in \R$. Further recall $\hat f(t)=\frac{\pi}{(e^{\frac{\pi t}{2}}+e^{-\frac{\pi t}{2}})^2}$ and $f_s(x)= \int_\R \hat f(t)e^{isxt} dt$ from Lemma \ref{ftsechsq}. Thus
$$f_s(A)v=\int_\R \hat f(t)\gamma(st)v dt$$
for $s\in\R, v\in V^\cO$ as a weak integral in the Hilbert space $V_A$. By Remark \ref{equicontsomermks}(b), this equation also holds as a weak integral in $V^\cO$ w.r.t. the $C^\cO$-topology. Let $\lambda\in (V^\cO)'$ be a continuous linear form on $V^\cO$. Then
\begin{align}\label{pdlem15eq1}
\lambda(f_s(A)v)= \int_\R u(t,s) dt, \quad u(t,s):= \hat f(t)\lambda(\gamma(st)v).
\end{align}
Since $\hat f$ and $\gamma:\R\times V^\cO\rightarrow V^\cO$ are smooth (Proposition \ref{equicontholvecprop}(c)), the map $u:\R\times \R\rightarrow V^\cO$ is smooth. Since $\lambda\in (V^\cO)'$ and the topology on $V^\cO$ is generated by $\gamma$-invariant seminorms, there is a $\gamma$-invariant seminorm $q$ on $V^\cO$ such that $|\lambda(v)|\leq q(v), v\in V^\cO$. Then we have
$$\frac{\partial^k u}{\partial s^k}(t,s)=\hat f(t)t^k\cdot \lambda(\gamma(st)D^kv) \quad \text{ and } \quad |\lambda(\gamma(st)D^kv)|\leq q(D^kv).$$
Since $\R\ni t\mapsto \hat f(t)t^k$ is an integrable function for every $k\in \N_0$ we obtain by differentiation under the integral sign that the map $\R\ni s\mapsto \lambda(f_s(A)v)$ is smooth with derivatives 
$$\frac{\partial^k \lambda(f_s(A)v)}{\partial s^k}(t,s)=\int_\R \frac{\partial^k u}{\partial s^k}(t,s) dt=\lambda\left(\int_\R\hat f(t)t^k\gamma(st)D^kvdt\right).$$
Here we note that the integral $\int_\R\hat f(t)t^k\gamma(st)D^kvdt$ exists as a weak integral in $V^\cO$ by Remark \ref{equicontsomermks}(b). By Grothendieck's Theorem (\cite[Thm.~D.1]{Gl04}) we conclude that $\R\ni s\mapsto \alpha(s,v)=f_s(A)v$ is smooth with derivatives
$$\frac{\partial^k \alpha}{\partial s^k}(s,v)=\int_\R\hat f(t)t^k\gamma(st)D^kvdt.$$
For a $\gamma$-invariant seminorm $q$ on $V^\cO$ we thus have
\begin{align*}
q\Big(\textstyle{\frac{\partial^k \alpha}{\partial s^k}}(s,v)-\textstyle{\frac{\partial^k \alpha}{\partial s^k}}(s_n,v_n)\Big)\leq
q\Big(\textstyle{\frac{\partial^k \alpha}{\partial s^k}}(s,v)-\textstyle{\frac{\partial^k \alpha}{\partial s^k}}(s_n,v)\Big)+\int_\R\hat f(t)t^kdt\cdot q(D^k(v-v_n)).
\end{align*}
This shows that $\frac{\partial^k \alpha}{\partial s^k}$ is continuous for all $k\in\N_0$, since we have already shown that $\R\ni s\mapsto \alpha(s,v)$ is smooth for all $v\in V^\cO$.	As $\alpha(s,v)$ is linear in $v$ for fixed $s$, we conclude that $\alpha$ is smooth.
\end{proof}

Now we want to express $\exp_\C(0,ix,is)$ in terms of $f_s(A)$ and $g_s(A)$. Therefore we consider the following map.

\begin{dfn}
We define the map $\theta: V^\cO \times \R \rightarrow \Heis(V_\C^\cO,\omega_\C)\subset G_{A,\C}$ by
$$\theta(v,s): = \left(\frac{i}{2s}\|f_s(A)v\|^2-\frac{i}{2s}\Re \langle f_s(A)v,v \rangle\ ,\  -g_s(A)v-iv\ ,\ 0\right)$$
for $s\neq 0$ and $\theta(v,0):=(0,-iv,0)$.
\end{dfn}

\begin{lem}\label{polardecomplemma2}
\begin{itemize}
\item[\rm(a)] $\exp_\C(0,ix,is)\exp_\C(0,0,-is)= \theta(f_s(A)^{-1}x,s)^{-1}$ holds for all $x\in V^\cO$ and $s\in \R$.
\item[\rm(b)] The map $\theta\vert_{V^\cO \times \R^\times}$ is analytic. Moreover $\theta\vert_{V^\cO \times \R^\times}$ is analytic when $V^\cO$ and $V^\cO_\C$ are equipped with the $C^\infty$-topology. Furthermore $\theta\vert_{V^\cO \times \R^\times}$ is analytic when $V^\cO$ and $V^\cO_\C$ are equipped with the topology induced by the norm $x\mapsto \|x\|+\|Ax\|, x \in V^\cO$.
\end{itemize}
\end{lem}
\begin{proof}
Recall $D=IA$ on $V^\cO$.  Note that
$$g_s(A)f_s(A)^{-1} = \frac{e^{sA}+e^{-sA}-2}{2sIA}.$$
Let $x\in V^\cO,s\in \R^\times$. Recall \eqref{polardecomplemeq1}:
$$B_{is}ix=g_s(A)f_s(A)^{-1}x+if_s(A)^{-1}x.$$
Thus we obtain:
\begin{align}
\frac{1}{is}\langle ix,\sigma(B_{is}ix-ix)\rangle_\C& = \frac{1}{s}\langle x,g_s(A)f_s(A)^{-1}x-if_s(A)^{-1}x\rangle_\C+\frac{1}{is}\|x\|^2 \nonumber \\
&= \frac{i}{s}\left(\Re\langle x,f_s(A)^{-1}x\rangle - \|x\|^2\right). \label{polardecomplemma2eq1}
\end{align}
Here the last equation follows from the fact that $g_s(A)f_s(A)^{-1}$ is skew-symmetric on $V^{\cO}$, which is easily seen from the definition of $f_s(A)$ and $g_s(A)$ for $s\in \R$, cf. Definition~\ref{deff(A)g(A)}. By \eqref{expformulaspec} we have:
$$\exp_\C(0,ix,is)\exp_\C(0,0,-is)= \left(\frac{1}{2is}\left\langle ix,\sigma(B_{is}ix-ix) \right\rangle_\C \ ,\ B_{is}ix\ ,0\right).$$
This equation together with \eqref{polardecomplemma2eq1} and \eqref{polardecomplemeq1} yields (a) for $s\neq 0$. Obviously (a) also holds for $s=0$. Note that (b) is a consequence of Lemma \ref{polardecomplemma1}(a).
\end{proof}
We also record the following lemma which we will need later.

\begin{lem}\label{swsemigrplem0}
Let $x=(x_0,0,x_2)\in \t=\R\times\{0\}\times \R, x_2\neq 0$. Then for every $y\in \g_A^\cO$ there exists $a\in \overline{[x,\g_A^\cO]}$, such that the logarithmic derivative of $\exp_\C$ satisfies
\begin{align}\label{swsemigrplem0eq1}
\delta(\exp_\C)_{ix}(i(y+a)) \in iy+\g_A^\cO.
\end{align}
\end{lem}
\begin{proof}
Note $\overline{[x,\g_A^\cO]}= \{0\}\times V^\cO\times \{0\}$ since $D(V^\cO)\subset V^\cO$ is dense. Let $a=(0,a_1,0)\in\overline{[x,\g_A^\cO]}$ and $y=(y_0,y_1,y_2)\in \g_A^\cO$. Then by \cite[Prop. II.5.7]{Ne06} and \eqref{polardecomplemeq1}
\begin{align*}
\delta(\exp_\C)_{ix}(i(y+a)) &= \int_0^1 \Ad(-tix)i(y+a)\ dt= \left(iy_0, \int_0^1 \gamma_\C(-tix_2)i(y_1+a_1)\ dt,iy_2\right) \\
&= \left(iy_0,B_{-ix_2}i(y_1+a_1),iy_2\right) \\
&= \left(iy_0,\ g_{-x_2}(A)f_{-x_2}(A)^{-1}(y_1+a_1)+if_{-x_2}(A)^{-1}(y_1+a_1),\ iy_2\right)
\end{align*}
Thus in order to satisfy \eqref{swsemigrplem0eq1} we may choose $a_1:=f_{-x_2}(A)y_1-y_1$. Actually one may show that then even $(0,a_1,0)\in [x,\g_A^\cO]$.
\end{proof}

For $v=x+iy\in V_\C$ with $x,y\in V$ we set $\Re(v):=x$ and $\Im(v):=y$.

\begin{prop}[Polar decomposition for $G_{A,\C}$]\label{polarmapG_C}
\begin{itemize}
\item[\rm(a)] The map 
$$\psi: G_A^{\cO} \times \g_A^\cO \rightarrow G_{A,\C}, (g,w) \mapsto g\exp_{\C}(iw)$$
is a diffeomorphism and $\psi$ is analytic.
\item[\rm(b)] The map $\vp: G_{A,\C} \rightarrow G_A^\cO \times \g_A^\cO, \vp(z,y,r)=(g,(t,x,s))$ defined by 
$$s=\Im(r),\ \ x=\gamma(-\Re(r))f_s(A)(\Im(y))$$
and
\begin{align}\label{polardecprfeq1G_C}
g\cdot(it,0,0)=(z,y,0)\cdot\theta(\Im(y),\Im(r))\cdot(0,0,\Re(r))
\end{align}
is the inverse of $\psi$.
\end{itemize}
\end{prop}
\begin{proof}
First note that with the explicit formula for the multiplication in $G_{A,\C}$ and the definition of $\theta$ one sees that the right hand side of \eqref{polardecprfeq1G_C} lies indeed in $\C\times V^\cO\times \R$. Hence \eqref{polardecprfeq1G_C} defines $g$ and $t$ in terms of $(z,y,r)$. Now let $g\in G_A^\cO, (t,x,s)\in \g_A^\cO$ and set $$(z,y,r):=\psi(g,(t,x,s))=g\exp_\C(it,ix,is).$$
With \eqref{expformulaspec} and \eqref{polardecomplemeq1} (and an extra consideration if $s=0$) we obtain
$$(z,y,r)=g\cdot\left(u,\ g_s(A)f_s(A)^{-1}x+if_s(A)^{-1}x,\ is\right)$$
for some $u\in\C$ which we do not specify. This equation together with the explicit formula for the multiplication in $G_{A,\C}$ implies 
$$\Im(r)=s \quad \text{ and } \quad \Im(y)=\gamma(\Re(r))f_s(A)^{-1}x.$$
With these equations and Lemma \ref{polardecomplemma2}(a) we further obtain
\begin{align*}
g\cdot(it,0,0) &=(z,y,r)\cdot\exp_\C(0,ix,is)^{-1} \\
&= (z,y,is)\cdot\exp_\C\left(\Ad(0,0,\Re(r))(0,ix,is)\right)^{-1}\cdot(0,0,\Re(r))\\
&=(z,y,0)\cdot(0,0,is)\cdot\exp_\C(0,i\gamma(\Re(r))x,is)^{-1}\cdot(0,0,\Re(r))\\
& =(z,y,0)\cdot\theta(f_s(A)^{-1}\gamma(\Re(r))x,s)\cdot(0,0,\Re(r))\\
&=(z,y,0)\cdot\theta(\Im(y),\Im(r))\cdot(0,0,\Re(r)).
\end{align*}
This shows that $\vp \circ \psi = \id$. From Lemma \ref{polardecomplemma2}(a) and the definition of $\vp$ one easily derives $\psi\circ \vp = \id$, so that $\psi$ is bijective with inverse $\vp$. Certainly the map $\psi$ is analytic. To show that $\vp=\psi^{-1}$ is smooth we argue as follows: Let 
$$\psi(g,(t,x,s)) = g\exp_\C(it,ix,is)=q\in G_{A,\C}.$$
We must show that $g,t,x,s$ depend smoothly on $q$. Since the action $\gamma: \R\times V^\cO \rightarrow V^\cO$ is smooth we obtain from the definition of $\vp=\psi^{-1}$ and Lemma \ref{polardecomplemma1.5} that both $s$ and $x$ depend smoothly on $q$. Since $g\exp_\C(it,0,0)=q\exp_\C(0,-ix,-is)$ we conclude that $g,t,x,s$ depend smoothly on $q$, i.e. that $\psi^{-1}$ is smooth.
\end{proof}

\begin{prop}
If $A$ is not bounded then the  map $\psi^{-1}:G_{A,\C}\rightarrow G_A^{\cO} \times \g_A^\cO $ from Proposition \ref{polarmapG_C} is not analytic.
\end{prop}
\begin{proof}
For $r>0$ we set $B_r:=\{z \in \C : |z|<r\}$. Suppose that $\psi^{-1}$ is analytic. By Proposition \ref{polarmapG_C}(b) the map
$$\alpha : \R \times V^\cO \rightarrow V^\cO, (s,v) \mapsto f_s(A)v$$
is then also analytic. Hence there exists $R>0, U\subset V^\cO$ an open $0$-neighborhood, $C_1>0$ and a holomorphic map $\tilde \alpha: B_R \times U_\C \rightarrow V^\cO$ such that 
$$\tilde\alpha \vert_{(B_R\cap \R) \times U} = \alpha\vert_{(B_R \cap \R)\times U}$$
and $\|\tilde\alpha(a)\|\leq C_1$ holds for all $a\in B_R \times U_\C$. Here $U_\C=U+iU$. Since $A$ is not bounded and $A\geq0$ we find $x_0>\frac{4}{R}$ such that $x_0\in \spec(A)$, i.e. $P_A([x_0-\eps,x_0+\eps])V\neq \{0\}$ for all $\eps >0$. Recall that the $C^\cO$-topology on $V^\cO$ is generated by the norms $q_n(v)=\sum_{k\geq0}\frac{n^k}{k!}\|A^kv\|, n\in \N$. Since $q_n(v)\leq q_{n'}(v)$ for $n\leq n'$, we find $N\in \N$ and $N'>0$ such that 
$$\{ v \in V^\cO : q_N(v) < N'\} \subset U.$$
For $v\in P_A([x_0-1,x_0+1])V$ we have $q_N(v) \leq e^{N(x_0+1)}\|v\|$. Thus there is $C_2>0$ such that 
$$(P_A([x_0-1,x_0+1])V) \cap \{v\in  V^\cO : \|v\| \leq C_2\} \subset U.$$ 
Recall the holomorphic map
$$\tilde{f} :\{0\} \cup \C \backslash \pi i\Z \rightarrow \C, \tilde{f}(z) = \frac{2z}{e^{z}-e^{-z}}.$$
Now choose $0<\eps<\min(1,x_0)$, $t_0\in ]0,R[$ and $0<\delta <t_0$ such that 
\begin{align}
tx<\pi \quad &\text{ for all } t\in [0,t_0], x\in [x_0-\eps,x_0+\eps], \label{psiinvnotanalyteq1}\\
|\tilde{f}(itx)| \geq \frac{2C_1}{C_2} \quad &\text{ for all }  t\in [t_0-\delta,t_0], x \in ]x_0-\eps,x_0+\eps[\label{psiinvnotanalyteq2}.
\end{align}
This is possible since $x_0>\frac{4}{R}>\frac{\pi}{R}$ and $\lim_{z\rightarrow \pi} |\tilde{f}(iz)|=\infty$. Now choose 
$$v_0\in P_A([x_0-\eps,x_0+\eps])V \ \text{ with } \ \|v_0\|= C_2.$$
In particular $v_0\in U$. Consider the holomorphic map
$$h: B_{t_0} \rightarrow V^\cO,h(s) = \int_{x_0-\eps}^{x_0+\eps} \tilde f(sx) \ dP_A(x)v_0,$$
which is well-defined by \eqref{psiinvnotanalyteq1}. Certainly $h(s)= f_s(A)v_0$ for all $s\in B_{t_0}\cap \R$ and thus the Identity Theorem yields $h(s) = \tilde \alpha(s,v_0)$ for all $s\in B_{t_0}$. Hence we obtain with \eqref{psiinvnotanalyteq2}:
$$\|\tilde\alpha(i(t_0-\delta),v_0)\|^2 = \int_{x_0-\eps}^{x_0+\eps} |\tilde f(i(t_0-\delta)x)|^2 d\langle P_A(x)v,v \rangle \geq \frac{4C_1^2}{C_2^2}C_2^2 = 4C_1^2.$$
This is a contradiction to $\|\tilde\alpha(i(t_0-\delta),v_0)\| \leq C_1$.
\end{proof}

Recall the open invariant cone $W_\infty^\cO=\R \times V^\cO \times ]0,\infty[\subset \g_A^\cO$ and the complex subsemigroup $S_A$ of $G_{A,\C}$ from Definition \ref{complexsemigroupsG_A}(c).

\begin{prop}[Polar decomposition for $S_A$]\label{polarmap}
\begin{itemize}
\item[\rm(a)] The map 
$$\psi: G_A^{\cO} \times W_\infty^\cO \rightarrow S_A, (g,w) \mapsto g\exp_{\C}(iw)$$
is a diffeomorphism and both $\psi$ and $\psi^{-1}$ are analytic.
\item[\rm(b)] The map $\vp: S_A \rightarrow G_A^\cO \times W_\infty^\cO, \vp(z,y,r)=(g,(t,x,s))$ defined by 
$$s=\Im(r),\ \ x=\gamma(-\Re(r))f_s(A)(\Im(y))$$
and
\begin{align}\label{polardecprfeq1}
g\cdot(it,0,0)=(z,y,0)\cdot\theta(\Im(y),\Im(r))\cdot(0,0,\Re(r))
\end{align}
is the inverse of $\psi$.
\item[\rm(c)] The map $\psi^{-1}=\vp$ is also analytic when both $S_A$ and $G_A^\cO \times W_\infty^\cO$ are equipped with the topology induced by the $C^\infty$-topology on $V^\cO$. Moreover $\psi^{-1}$ is analytic when $G_A^\cO \times W_\infty^\cO$ is equipped with the topology induced by the norm-topology on $V^\cO$ and $S_A$ is equipped with the topology induced by the norm $x\mapsto \|x\|+\|Ax\|$ on $V^\cO$.
\end{itemize}
\end{prop}
\begin{proof}
Note that $\psi$ and $\vp$ are the restrictions of the corresponding maps in Proposition~\ref{polarmapG_C}. We denote these corresponding maps by $\psi_F$ and $\vp_F$ in this proof, i.e. $\psi$ is the restriction of $\psi_F$ and $\vp$ the restriction of $\vp_F$. Note that $\vp$ is defined since $\vp_F(S_A)\subset G_A^\cO \times W_\infty^\cO$. From Proposition~\ref{polarmapG_C} we obtain that $\psi$ is a diffeomorphism and that (b) holds. Certainly the map $\psi$ is analytic. The definition of $\vp=\psi^{-1}$, Lemma~\ref{polardecomplemma1}(a) and Lemma~\ref{polardecomplemma2}(b) imply that $\psi^{-1}$ is also analytic. Moreover from Lemma~\ref{polardecomplemma1}(b) and Lemma~\ref{polardecomplemma2}(b) we obtain that $\psi^{-1}=\vp$ is also analytic when both $S_A$ and $G_A^\cO \times W_\infty^\cO$ are equipped with the topology induced by the $C^\infty$-topology on $V^\cO$. Now let $V_1:=V$ be $V$ equipped with the topology given by the norm $x\mapsto \|x\|+\|Ax\|$ and $V_2:=V$ be $V$ equipped with the $\|\cdot\|$-topology. Then the map 
\begin{multline*}
\Heis(V_1)\times \Heis(V_1) \rightarrow \Heis(V_2),\\((t,x),(t',x'))\mapsto (t,x)\cdot(t',x')=\big(t+t'+\textstyle{\frac{1}{2}}\Im\langle Ax,x'\rangle,x+x'\big)
\end{multline*}
is analytic. This observation, Lemma~\ref{polardecomplemma1}(b) and Lemma~\ref{polardecomplemma2}(b) imply that $\psi^{-1}=\vp$ is also analytic when $G_A^\cO \times W_\infty^\cO$ is equipped with the topology induced by the norm-topology on $V^\cO$ and $S_A$ is equipped with the topology induced by the norm $x\mapsto \|x\|+\|Ax\|$.
\end{proof}

\begin{prop}\label{involutiong_a}
The map 
$$*:S_A \rightarrow S_A,g\exp_\C(ix) \mapsto g^{-1}\exp_\C(i\Ad(g)x)$$ 
defines an involution and turns $S_A$ into an involutive complex semigroup.
\end{prop}
\begin{proof}
The map $*$ is well-defined and analytic by Proposition \ref{polarmap}. Since $g^{-1}\exp_\C(i\Ad(g)x)=\exp_\C(ix)g^{-1}$ the map $*$ is the restriction of the map $*:G_{A,\C} \rightarrow G_{A,\C}$ from Proposition \ref{involutiononGAC}. Now the assertion follows from Proposition \ref{involutiononGAC}.
\end{proof}

\subsection{The complex semigroups $S_d$}\label{sectdefsd}
Recall the open invariant cones $W_d$ in $\g_A=\heis(V,\omega_A)\rtimes_D\R$, where $V:=C^\infty(A)$ and $C_d=W_d\cap \t$ in $\t=\R\times \{0\} \times \R$ for $d\in \overline{\R}$ from Section \ref{sectionopeninvariantcones}. We set $W_d^\cO=W_d \cap \g_A^\cO$ for all $d\in \overline{\R}$. From Proposition \ref{polarmap} we know that 
$$S_d:=G_A^\cO\exp_\C(iW_d^\cO)$$
is open in $S_A=S_\infty$ and the map $G_A^\cO\times W_d^\cO \rightarrow S_d, (g,w) \mapsto g\exp_\C(iw)$ is an analytic diffeomorphism. We now want to see that $S_d$ is a subsemigroup of $G_{A,\C}$ for every $d\in \overline{\R}$. To do so we adapt the proof of \cite[Theorem XI.1.10]{Ne00} to our infinite-dimensional oscillator group.

For $d\in \overline{\R}$ choose a smooth function $\tilde{h}_d:C_d\rightarrow \R$ such that:
\begin{itemize}
\item[\rm(a)] $\tilde{h}_d(x_n)\rightarrow \infty$ whenever $x_n\rightarrow x\in \partial C_d \subset \t$.
\item[\rm(b)] $\tilde{h}_d(x+y)\leq \tilde{h}_d(x)$ for $x,y\in C_d$.
\end{itemize}
See also \cite[Thm. VIII.3.16]{Ne00} for the existence of such a function. Note that $(t,x,s)\in W_d$ implies $(t-\frac{1}{2s}\|x\|^2,0,s)\in C_d$ which may be seen from the explicit definition of $W_d$. Now we define the smooth functions
$$h_d:W_d\rightarrow \R, \quad h_d(t,x,s):=\tilde{h}_d\Big(t-\textstyle{\frac{1}{2s}}\|x\|^2,0,s\Big).$$

\begin{lem}\label{swsemigrplem1}
For every $d\in \overline{\R}$, the smooth map $h_d$ has the following properties:
\begin{itemize}
\item[\rm(a)] $h_d$ is $\Ad$-invariant.
\item[\rm(b)] $h_d(w_n)\rightarrow \infty$ whenever $w_n\rightarrow w\in \partial W_d \subset \g_A$.
\item[\rm(c)] $h_d(w+w')\leq h_d(w)$ for $w,w'\in W_d$.
\end{itemize}
\end{lem}
\begin{proof}
(a) Since
\begin{align}\label{adjointactionont}
\Ad(t',x',s')(t,0,s)=\Big(t+\textstyle{\frac{1}{2}}s\|Dx'\|^2,-sDx',s\Big)
\end{align}
by \eqref{adjointGA}, the definition of $h_d$ shows that $h_d$ is constant on $\overline{\Ad(G_A)a}$ for every $a\in C_d$. From the definition of $W_d$ we see that $\bigcup_{a\in C_d}\overline{\Ad(G_A)a} = W_d$. Now (a) follows.\newline
(b) We write $w=(t',x,s)$ and by Definition \ref{Wddef} we may write
$$w_n=\Big(t_n+\textstyle{\frac{1}{2s_n}}\|x_n\|^2,x_n,s_n\Big) \quad \text{ with $t_n+ds_n>0, s_n>0$. }$$
Suppose first that $s>0$. Since $w_n\rightarrow w$, the sequence $(t_n)_n$ converges in this case, say $t_n\rightarrow t$ so that $t'=t+\frac{1}{2s}\|x\|^2$. Since $w\notin W_d$ we must have $t+ds=0$ which implies $(t,0,s)\in \partial C_d$. Hence $h_d(w_n)=\tilde{h}_d(t_n,0,s_n)\rightarrow \infty$ by the choice of $\tilde{h}_d$. Now suppose $s=0$. Since $t_n+ds_n>0$ and $s_n\rightarrow 0$ the sequence $(t_n)_n$ is bounded below. Moreover $t_n\leq t_n+\frac{1}{2s_n}\|x_n\|^2\rightarrow t'$ implies that $t_n$ is also bounded above. Thus we find a subsequence $(t_{n_k})_k$ such that $t_{n_k}\rightarrow t$ and 
\begin{align}\label{swsemigrplem1eq0}
h_d(w_{n_k})\rightarrow \liminf_n(h_d(w_n))\qquad \text{ as $k\rightarrow \infty$}.
\end{align}
Since $h_d(w_{n_k})=\tilde{h}_d(t_{n_k},0,s_{n_k})$ and $(t_{n_k,}0,s_{n_k})\rightarrow (t,0,0)\in \partial C_d$, we conclude that the sequence $(h_d(w_{n_k}))_k$ converges to $\infty$. Hence $h_d(w_n)\rightarrow \infty$ by \eqref{swsemigrplem1eq0}.\newline
(c) Write $w=(t+\frac{1}{2s}\|x\|^2,x,s)$ and $w'=(t'+\frac{1}{2s'}\|x'\|^2,x',s')$. Then
\begin{align}\label{swsemigrplem1eq1}
h_d(w+w') = \tilde{h}_d\big(t+t'+\delta,0,s+s'\big)=\tilde{h}_d\big((t,0,s)+(t'+\delta,0,s')\big)
\end{align}
where $\delta:=\frac{1}{2s}\|x\|^2+\frac{1}{2s'}\|x'\|^2-\frac{1}{2(s+s')}\|x+x'\|^2$. By convexity of the map $x\mapsto \|x\|^2$, we have
$$\textstyle{\frac{s}{s+s'}}\big\|\textstyle{\frac{1}{s}}x\big\|^2+\textstyle{\frac{s'}{s+s'}}\big\|\textstyle{\frac{1}{s'}}x'\big\|^2\geq \left\|\textstyle{\frac{x+x'}{s+s'}}\right\|^2 =\textstyle{\frac{\|x+x'\|^2}{(s+s')^2}}.$$
Hence $\delta\geq 0$. Note that $(t,0,s),(t',0,s')\in C_d$. Since $\R_{\geq0}\times\{0\} \times \{0\} \subset \lim(C_d)$, we thus obtain $(t'+\delta,0,s')\in C_d$. Hence equation \eqref{swsemigrplem1eq1} and the choice of $\tilde{h}_d$ imply $h_d(w+w')\leq h_d(w)$.
\end{proof}

\begin{lem}\label{swsemigrplem2}
Let $f:W_d^\cO\rightarrow \R$ be an $\Ad$-invariant smooth function and set $F:S_d\rightarrow\R, F(g\exp_\C(ix))=f(x)$. Then
$$dF(s)(s.(iy))= df(x)y$$
holds for $s=g\exp_\C(ix)\in S_d, y\in \g_A^\cO$.
\end{lem}
\begin{proof}
Let $G:=G_A^\cO, \g:=\g_A^\cO$. Note that $S_d$ is invariant under both left and right multiplication by $G$, since $W_d^\cO$ is $\Ad$-invariant. Note also that $F$ is $G$-biinvariant since $f$ is $\Ad_{G}$-invariant. The left $G$-invariance $F\circ \lambda_g =F$ yields for $s=g\exp_\C(ix)$
$$dF(s)(s.(iy))=dF(g\exp_\C(ix))(g.(\exp_\C(ix).(iy)))=dF(\exp_\C(ix))(\exp_\C(ix).(iy)).$$
Thus it suffices to show that
\begin{align}\label{swsemigrplem2redeq}
dF(\exp_\C(ix))(\exp_\C(ix).(iy))= df(x)y
\end{align}
holds for all $x\in W_d^\cO,y\in \g$. The $\Ad$-invariance of $f$ and the right $G$-invariance of $F$ show
\begin{align}\label{swsemigrplem2eq1}
df(x)\big(\overline{[\g,x]}\big)=\{0\} \quad \text{ and } \quad dF(s)(s.\g) =\{0\}.
\end{align}
Since
$$df(x)y=dF(\exp_\C(ix))(\exp_\C(ix).\delta(\exp_\C)_{ix}(iy))$$
we obtain from Lemma \ref{swsemigrplem0} and \eqref{swsemigrplem2eq1} that \eqref{swsemigrplem2redeq} holds for all $x\in C_d,y\in \g$. Next we note that both the left and the right hand side of \eqref{swsemigrplem2redeq} do not change, if both $x$ and $y$ are replaced by $\Ad(g)x$ resp. $\Ad(g)y$ for some $g\in G$. This follows from the invariance properties of $f$ and $F$. Thus \eqref{swsemigrplem2redeq} holds for all $x\in \Ad(G)C_d,y\in \g$. By continuity \eqref{swsemigrplem2redeq} therefore holds for all $x\in \overline{\Ad(G)C_d}\cap W_d^\cO$ and $y\in \g$. Here the closure is taken w.r.t. the $C^\cO$-topology.  Since $D(V^\cO)\subset V^\cO$ is dense, the definition of $W_d$ and \eqref{adjointactionont} show that $\Ad(G)C_d$ is dense in $W_d^\cO$ w.r.t. the $C^\cO$-topology. Hence \eqref{swsemigrplem2redeq} holds for all $x\in W_d^\cO,y\in \g$. This completes the proof.
\end{proof}

\begin{prop}\label{G_Acomplexifiable}
For every $d\in \overline{\R}$, the pair $(G_A^\cO,W_d^\cO)$ is complexifiable (see Definition \ref{complexifiablepairdef} below) with corresponding complex involutive semigroup $S_d$.
\end{prop}
\begin{proof}
From Proposition \ref{polarmap} we know that $S_d=G_A^\cO\exp_\C(iW_d^\cO)$ is open in $G_{A,\C}$ and that the map $G_A^\cO\times W_d^\cO \rightarrow S_d, (g,x) \mapsto g\exp_\C(ix)$ is an analytic diffeomorphism. Note also that $S_d$ is invariant under the involution $*$ on $S_A$ from Proposition \ref{involutiong_a}. Thus we only have to show that $S_d$ is a subsemigroup of $G_{A,\C}$. Let $x,y\in W_d^\cO$. We consider the curve $c(t):=\exp_\C(ix)\exp_\C(tiy)$ for $t\in \R_{\geq0}$. Note that this curve does not leave the semigroup $S_A=S_\infty$ so that we may write $c(t)=g_t\exp(ix_t)$ for all $t$ according to the polar decomposition from Proposition \ref{polarmap}. Recall the function $h_d$ from Lemma \ref{swsemigrplem1}. We set $H_d(g\exp_\C(iw)):=h_d(w)$ for $g\exp_\C(iw)\in S_d$. By Lemma \ref{swsemigrplem2}, we have
$$(H_d\circ c)'(t) = d H_d(c(t)) (c(t).iy)= d h_d(x_t)(y) \leq 0$$
as long as $c(t)\in S_d$, where the last inequality follows from Lemma \ref{swsemigrplem1}(c). Hence the function $H_d\circ c$ is decreasing. Thus Lemma \ref{swsemigrplem1}(b) implies that the curve does not leave $S_d$. In particular $c(1)\in S_d$. This proves $\exp_\C(ix)\exp_\C(iy)\in S_d$ for all $x,y\in W_d^\cO$. Since 
$$g\exp_\C(ix)g'\exp_\C(ix')=gg'\exp_\C(i\Ad(g'^{-1})x)\exp_\C(ix'),$$
we conclude that $S_d$ is a subsemigroup.
\end{proof}

\section{Extensions of representations of general Lie groups}\label{extensionsofrepofgenliegroups}

In this section we discuss holomorphic extensions of semibounded representations for general Lie groups. Let $G$ be a simply-connected locally convex analytic Lie group with exponential map $\exp:\g \rightarrow G$ and assume that $\g$ is Mackey complete. Let $W\subset \g$ be an $\Ad$-invariant open cone in $\g$.

\begin{dfn}\label{complexifiablepairdef}
The pair $(G,W)$ is called \textit{complexifiable} if there exists a complex Lie group $G_\C$ with Lie algebra $\g_\C$, an exponential map $\exp_\C:\g_\C \rightarrow G_\C$ and a Lie group embedding $G\subset G_\C$ which induces the inclusion $\g\hookrightarrow \g_\C$ of Lie algebras, such that $S_W:=G\exp_\C(iW)$ is an open subsemigroup of $G_\C$,
$$\psi: G \times W \rightarrow S_W, (g,w) \mapsto g\exp_\C(iw)$$
is an analytic diffeomorphism, i.e. $\psi$ and $\psi^{-1}$ are both analytic, and 
$$*:S_W\rightarrow S_W, g\exp_\C(iw) \mapsto \exp_\C(iw)g^{-1}=g^{-1}\exp_\C(i\Ad(g)w)$$
turns $S_W$ into an involutive complex semigroup.
\end{dfn}

\begin{rmk}
In the situation of the preceding definition suppose that there is a group automorphism $\tau: G_\C \rightarrow G_\C$ which is antiholomorphic, involutive and satisfies $\L(\tau)(x+iy)=x-iy$ for $x,y\in \g$. Then $g^*:=\tau(g)^{-1}$ defines an antiholomorphic involutive antiautomorphism of $G_\C$ which satisfies $(g\exp_\C(iw))=\exp_\C(iw)g^{-1}$. Hence the existence of $*:S_W\rightarrow S_W$ in Definition \ref{complexifiablepairdef} follows from the existence of $\tau$. In particular, if $G_\C$ is regular and simply-connected, then $\tau$ exists by \cite[Thm. III.1.5]{Ne06}.
\end{rmk}

\begin{thm}\label{holextgen}
Let $(G,W)$ be complexifiable and $\pi:G \rightarrow \U(\H)$ be a semibounded representation with $W \subset B(I_\pi)^0$. Then the map 
$$\hat{\pi}: S_W \rightarrow B(\H), g\exp_\C(ix) \rightarrow \pi(g)e^{i\dd\pi(x)}$$ 
is a holomorphic representation of the involutive complex semigroup $S_W$.
\end{thm}

Theorem \ref{holextgen} was proved in \cite[Thm.~5.7]{MN11} when $G$ is assumed to be a Banach-Lie group. Actually, most of the arguments in the proof of \cite[Thm.~5.7]{MN11} carry over to our more general context. Let us give an outline of the proof of Theorem \ref{holextgen}. At first we show that \cite[Lemma 5.2]{MN11} also holds if $G$ is not assumed to be Banach.

\begin{dfn}
\begin{itemize}
\item[(a)] Let $\g$ be a locally convex Lie algebra, $E$ be a Banach space and $\D \subset E$ a dense subspace. Let $\alpha : \g \rightarrow \End(\D)$ be a representation of $\g$. Then $\alpha$ is said to be \textit{strongly continuous}, if for every $v\in \D$ the map $\alpha^v: \g \rightarrow \D, x\mapsto \alpha(x)v$ is continuous.
\item[(b)] A locally convex Lie algebra is called \textit{$\ad$-integrable} if for every $x\in \g$ there is a one-parameter group $\Phi^x:\R\rightarrow \GL(\g)$ such that the corresponding action $\R\times \g \rightarrow \g$ is smooth and $\derat0 \Phi^x(t)y = \ad x(y)$ holds for all $x,y\in \g$. We then use the notation $e^{t\ad x}:=\Phi^x(t)$, see also \cite[Def.~B.2]{MN11}.
\end{itemize}
\end{dfn}

\begin{lem}
Let $G$ be a locally convex Lie group with an exponential map such that $\g$ is Mackey complete and $\g_\C$ is $\ad$-integrable. Let $\pi:G \rightarrow \U(\H)$ be a semibounded representation with $W \subset B(I_\pi)^0$. For every $v\in \H^\infty$ the map $\rho^v:W \rightarrow \H, x \mapsto e^{i\dd\pi(x)}v$ is $C^1$ with
$$d\rho^v(x)(y)= e^{i\dd\pi(x)}\dd\pi\left(\int_0^1e^{-s\ad ix}iy ds\right)v.$$
Furthermore the map $G\times W \rightarrow \H, (g,x) \mapsto \pi(g)e^{i\dd\pi(x)}v$ is also $C^1$.
\end{lem}
\begin{proof}
Consider the subspace
$$\D^\infty_\g:= \bigcap_{y_1,\dots,y_n\in \g, n \in \N} \D(\overline{\dd\pi}(y_n) \cdots \overline{\dd\pi}(y_1))$$
of $\H$ (cf. also \cite[Def. 3.1(d)]{Ne10a}). Note that $\H^\infty \subset \D^\infty_\g$. In particular $\D^\infty_\g$ is dense in $\H$. Obviously $\D^\infty_\g$ is invariant under $\overline{\dd\pi}(x)$ for all $x\in\g$. Let $v\in \D^\infty_\g$. From $\pi(\exp(tx))\pi(g)v=\pi(g)\pi(\exp(t\Ad(g)^{-1}x))v$ we obtain 
$$\pi(g)v \in \D(\overline{\dd\pi}(x))\ \ \text{ and }\ \ \overline{\dd\pi}(x)\pi(g)v=\pi(g)\overline{\dd\pi}(\Ad(g)^{-1}x)v.$$
Hence $\overline{\dd\pi}(x)\pi(g)(\D^\infty_\g) \subset \pi(g)(\D^\infty_\g)$. By induction we conclude that $\D^\infty_\g$ is invariant under $\pi(G)$. 

For $v\in \D^\infty_\g$ and $n\in \N_0$ we define
$$\omega_n^v:\g^n \rightarrow \H,\omega^v_n(x_1,\dots,x_n)=\overline{\dd\pi}(x_1)\cdots\overline{\dd\pi}(x_n)v.$$
We call $v\in \D^\infty_\g$ a $\textit{good vector}$ if $\omega_n^v$ is $n$-linear and continuous for all $n\in \N$ and 
$$\overline{\dd\pi}(x)\overline{\dd\pi}(y)\omega^v_n(x_1,\dots,x_n)-\overline{\dd\pi}(y)\overline{\dd\pi}(x)\omega^v_n(x_1,\dots,x_n)=\overline{\dd\pi}([x,y])\omega_n^v(x_1,\dots,x_n)$$
holds for all $x,y,x_1,\dots,x_n\in \g, n\in \N_0$. Let $\D$ denote the space of good vectors in $\D^\infty_\g$. Note that $\H^\infty \subset \D$. Note also that $\D$ is invariant under $\overline{\dd\pi(x)}$ for all $x\in \g$. We thus obtain a strongly continuous representation
$$\alpha:\g \rightarrow \End(\D),x \mapsto \overline{\dd\pi(x)}\vert_\D.$$
Complex-linear extension yields a strongly continuous representation $\alpha:\g_\C \rightarrow \End(\D)$.
Now one shows as in the proof of \cite[Lemma 5.2]{MN11} that
$$e^{i\dd\pi(x)} \D \subset \D^\infty_\g \ \ \text{ and } \ \ \overline{\dd\pi}(y)e^{i\dd\pi(x)}v=e^{i\dd\pi(x)}\overline{\dd\pi}(e^{-\ad ix}y)v$$
for all $v\in \D,x\in W,y\in \g_\C$. This further implies that $e^{i\dd\pi(x)} \D \subset \D$.
We may thus apply \cite[Prop. B.5]{MN11} to $\alpha$ which shows that
$$\hat{\rho}: W \times \H \rightarrow\H,(x,v) \mapsto e^{i\dd\pi(x)}v$$
is continuous and that for every $v\in \D$, $\rho^v$ is $C^1$ with differential as given. For $v\in \H^\infty$, consider the map
$$f: G \times W \rightarrow \H, (g,x) \mapsto e^{i\dd\pi(x)} \pi(g)v.$$
We want to show that $f$ is $C^1$. Since $v\in \H^\infty$ and $\rho^{\pi(g)v}$ is $C^1$ for each $g\in G$ it follows that the partial derivatives of $f$ exist both in $G$ and $W$ and are given by:
\begin{align*}
d_G f(g,x)(h) &= e^{i\dd\pi(x)} d(\pi^v)(g)(h) \\
d_W f(g,x)(y) &= e^{i\dd\pi(x)} \dd\pi\left(\int_0^1e^{-s\ad ix}iy ds\right)\pi(g)v \\
 & = e^{i\dd\pi(x)}\pi(g)\dd\pi\left(\Ad(g^{-1})\int_0^1e^{-s\ad ix}iy ds\right)v.
\end{align*}
Since $\hat{\rho}$ and the action of $\pi$ on $\H$ are both continuous, we see that these partial derivatives are continuous. Hence $f$ is $C^1$. Since
$$\pi(g)e^{i\dd\pi(x)}v=f(g,\Ad(g)x),$$
the map $G\times W \rightarrow \H, (g,x) \mapsto \pi(g)e^{i\dd\pi(x)}v$ is also $C^1$.
\end{proof}

With this lemma the proof of Theorem \ref{holextgen} now works the same way as the proof of \cite[Thm.5.4]{MN11} and \cite[Thm.5.7]{MN11}.
\begin{rmk}\label{pihatrmks}
The following assertions hold in the situation of Theorem \ref{holextgen}:
\begin{itemize}
\item[\rm(a)] $\pi(g)\hat\pi(s)=\hat\pi(gs)$ for all $g\in G, s\in S_W$.
\item[\rm(b)] The map $G\rightarrow B(\H),g\mapsto \pi(g)\hat\pi(s)$ is analytic for all $s\in S_W$ by (a) and the holomorphy of $\hat \pi$. In particular, the dense subspace $\spn(\hat\pi(S_W)\H)$ consists of analytic vectors for $\pi$.
\item[\rm(c)] $\hat\pi$ is non-degenerate. This can be seen as follows: Choose $w\in W$. The self-adjoint operator $B:=i\overline{\dd\pi(w)}$ is bounded above, say by $C>0$. Then for $v\in \H$ and $t\in [0,1]$ we have
$$\|\hat\pi(\exp_\C(iwt))v-v\|^2=\int_0^C (e^{xt}-1)^2 \ d\langle P_B(x)v,v\rangle,$$
where $P_B$ denotes the spectral measure of $B$. Here $(e^{xt}-1)^2\leq e^{2C}+1$ holds for the integrand. Thus $\lim_{t\rightarrow 0}\hat\pi(\exp_\C(iwt))v = v$ by the Dominated Convergence Theorem.
\end{itemize}
\end{rmk}

\begin{prop}\label{holextprop}
Let $(G,W)$ be complexifiable with corresponding complex involutive semigroup $S_W$ and let $\rho:S_W\rightarrow B(\H)$ be a non-degenerate holomorphic representation.
\begin{itemize}
\item[\rm(a)] The operator $\rho(s)$ is injective and has dense range for every $s\in S_W$.
\item[\rm(b)] There exists a unique smooth unitary representation $\pi:G\rightarrow\U(\H)$ such that $\pi(g)\rho(s)=\rho(gs)$ holds for all $g\in G, s\in S_W$.
\end{itemize}
\end{prop}
\begin{proof}
(a) This can be proven as in \cite[Lemma XI.3.11]{Ne00}: Since $\rho(s^*)=\rho(s)^*$, it suffices to show that $\rho(s)$ is injective for every $s\in S_W$. Let $\rho(s)v=0$. Then we have $\rho(S_Ws)v=\rho(S_W)\rho(s)v=\{0\}$. Since $S_Ws$ is open in $S_W\subset G_\C$, $S_W$ is connected and $\rho$ is holomorphic, we conclude with the Identity Theorem for Holomorphic Functions that $\rho(S_W)v=\{0\}$. Hence $v=0$ as $\rho$ is non-degenerate.\newline
(b) Note that the condition $\pi(g)\rho(s)=\rho(gs)$ determines $\pi$ uniquely since $\rho$ is non-degenerate. From \cite[Prop. 4.3]{Ne08} with $\alpha(s):=\|\rho(s)\|$ we obtain the existence of $\pi$ as a unitary representation. The condition $\pi(g)\rho(s)=\rho(gs)$ and the holomorphy of $\rho$ show that the dense subset $\spn( \pi(S)\H)$ consists of smooth vectors for $\pi$.
\end{proof}

\begin{lem}\label{holextlem}
Equip the semigroup $\C^+=\R+i\R_{\geq0}$ with the involution $(t+is)^*=-t+is$. Let $\rho:\C^+\rightarrow B(\H)$ be a $*$-representation such that $\rho\vert_{{\rm int}(\C^+)}$ is non-degenerate and holomorphic and such that $\rho\vert_\R$ is strongly continuous. Let $B$ denote the self-adjoint generator of $\rho\vert_\R$, i.e. $\rho(t)=e^{iBt}, t\in \R$. Then the following holds:
\begin{itemize}
\item[\rm(a)] $\rho$ is strongly continuous.
\item[\rm(b)] $B$ is bounded below.
\item[\rm(c)] $\rho(z)=e^{iBz}$ holds for every $z\in \C^+$ in the sense of functional calculus of $B$.
\end{itemize}
\end{lem}
\begin{proof}
(a) Note that $\rho(is)^*=\rho((is)^*) =\rho(is)$, i.e. $\rho(is)$ is self-adjoint for $s\in \R_{\geq 0}$. Thus $\|\rho\left(i\cdot\textstyle{\frac{n}{m}}\right)\|^m=\|\rho\left(i\cdot\textstyle{\frac{n}{m}}\right)^m\|=\|\rho(i)\|^n$ for $n,m\in \N$. This implies
$\|\rho(is)\|=\|\rho(i)\|^s$ for $s\in \R_{\geq0}$ since $\rho\vert_{\text{int}(\C^+)}$ is holomorphic. We conclude $\|\rho(t+is)\|=\|\rho(i)\|^s$ since $\rho(t)$ is unitary for $t\in \R$. In particular the map $\rho:\C^+\rightarrow B(\H)$ is locally bounded. For $y\in \text{int}(\C^+), v\in\H$ the map 
$$\C^+\rightarrow \H, z\mapsto \rho(z)\rho(y)v=\rho(z+y)v$$
is continuous, since $\rho\vert_{\text{int}(\C^+)}$ is holomorphic. As $\rho\vert_{\text{int}(\C^+)}$ is non-degenerate we obtain that $\C^+\ni z\mapsto \rho(z)v$ is continuous for $v$ in a dense subspace of $\H$. The local boundedness of $\rho$ now implies that $\rho$ is strongly continuous.\newline
(b) With (a) we obtain from \cite[Lemma XI.2.6]{Ne00} that $B$ is bounded below.\newline
(c) Part (b) implies $e^{iBz}\in B(\H)$ for every $z\in \C^+$. Now, for $v\in \H$, the maps $z\mapsto \rho(z)v$ and $z\mapsto e^{iBz}v$ are both continuous on $\C^+$, holomorphic on $\text{int}(\C^+)$ and they agree on~$\R$. Thus they must be equal.
\end{proof}

Let $(G,W)$ be complexifiable with corresponding complex involutive semigroup $S_W$. Recall $\hat\pi$ from Theorem \ref{holextgen}. Then we obtain the following correspondence.

\begin{thm}\label{sbrepholrepcorrespondencegen}
The assignment $\pi\mapsto \hat\pi$ yields a bijection
$$\Phi:\left\{
\begin{array}[c]{c}
\text{$\pi:G\rightarrow \U(\H)$ semibounded} \\ 
\text{representations with $W\subset B(I_\pi)^0$}
\end{array}
\right\}
\rightarrow
\left\{
\begin{array}[c]{c}
\text{$\rho:S_W\rightarrow B(\H)$ non-degenerate}\\
\text{holomorphic representations}
\end{array}
\right\}$$
which preserves commutants, i.e. $\pi(G)'=\hat\pi(S_W)'$.
\end{thm}
\begin{proof}
From Theorem \ref{holextgen} and Remark \ref{pihatrmks}(c) we know that $\Phi$ is defined. Suppose that $\Phi(\pi_1)=\Phi(\pi_2)$, i.e. $\hat\pi_1=\hat\pi_2$. Then $\pi_1(g)\hat\pi_1(s)=\hat\pi_1(gs)=\hat\pi_2(gs)=\pi_2(g)\hat\pi_1(s)$. Hence the unitary operators $\pi_1(g)$ and $\pi_2(g)$ coincide on the total subset $\pi_1(S_W)\H$ of $\H$ and must therefore be equal. Thus $\pi_1=\pi_2$, which shows that $\Phi$ is injective. Now let $\rho\colon S_W\rightarrow B(\H)$ be a non-degenerate holomorphic representation. By Proposition~\ref{holextprop}(b) there exists a smooth unitary representation $\pi:G\rightarrow\U(\H)$ with $\pi(g)\rho(s)=\rho(gs)$. For $x\in W$ we define
$$\nu_x: \C^+ \rightarrow B(H), z \mapsto \begin{cases} \rho(\exp_\C(zx)) &\text{ for $z\in\text{int}(\C^+)$} \\ \pi(\exp(zx)) &\text{ for $z\in \R$.}\end{cases}$$
Then $\nu_x$ is a representation since $\pi(g)\rho(s)=\rho(gs)$. Since $\rho$ is non-degenerate the operator $\rho(\exp_\C(ix))=\nu_x(i)$ has dense range by Proposition \ref{holextprop}(a). Therefore $\nu_x\vert_{\text{int}(\C^+)}$ is non-degenerate. Now we see that $\nu_x$ satisfies the conditions in Lemma \ref{holextlem}. Hence $-i\dd\pi(x)$ is bounded below and 
\begin{align}\label{sbcondholext}
\rho(\exp_\C(ix))=\nu_x(i)=e^{i\dd\pi(x)}
\end{align}
holds in the sense of functional calculus of $\overline{i\dd\pi(x)}$. Recall the support function $s_\pi(y)=\sup(\Spec(\overline{i\dd\pi(y)}))$ for $y\in \g$. From \eqref{sbcondholext} we obtain $\|\rho(\exp_\C(ix))\|=e^{s_\pi(x)}$ for all $x\in W$. Since $\rho:S_W\rightarrow B(\H)$ is holomorphic and thus in particular locally bounded, the map $W\ni x\mapsto e^{s_\pi(x)}$ is locally bounded. Hence $s_\pi$ is also locally bounded on $W$ which implies that $\pi$ is semibounded with $W\subset B(I_\pi)^0$. With \eqref{sbcondholext} we now obtain
$$\hat\pi(g\exp_\C(ix))=\pi(g)e^{i\dd\pi(x)}=\pi(g)\rho(\exp_\C(ix))=\rho(g\exp_\C(ix))$$
for all $g\in G, x\in W$. Hence $\rho=\hat\pi=\Phi(\pi)$ which shows the surjectivity of $\Phi$. That $\pi(G)'=\hat\pi(S_W)'$ holds may be shown as in \cite[Prop. XI.3.10]{Ne00}.
\end{proof}

\subsection*{Extending representations of dense subgroups}

In this subsection we will show that every smooth representation of a dense locally exponential Lie subgroup of a locally exponential Lie group $G$ is extendable to a smooth representation of $G$. This will be needed in the next section since the complex Lie group $G_{A,\C}$ associated to the oscillator group $G_A$ is the complexification of the dense subgroup $G_A^\cO \subset G_A$. At first we recall a continuous version for topological groups.

\begin{prop}\label{contextension}
Let $G$ be a topological group, $H\subset G$ a dense subgroup and let $\pi_H:H \rightarrow\U(\H)$ be a continuous unitary representation. Then $\pi_H$ is uniquely extendable to a continuous representation $\pi: G \rightarrow \U(\H)$.
\end{prop}
\begin{proof}
As $\pi_H$ is a direct sum of cyclic unitary representations (\cite[Prop.~II.2.11(ii)]{Ne00}), we may assume that $\pi_H$ is cyclic with cyclic vector $v\in\H$. Then the positive definite function $f_H:H\rightarrow \C, f_H(h)=\langle \pi_H(h)v,v\rangle$ is uniformly continuous w.r.t. the left uniform structure on $G$ (see \cite[Thm. 32.4(iv)]{HR63}). Hence $f_H$ is (uniquely) extendable to a continuous positive definite function $f:G\rightarrow \C$ which corresponds to a representation $\pi: G\rightarrow \U(\H)$ with cyclic vector $v$ and $\langle \pi(g)v,v\rangle=\langle \pi_H(g)v,v\rangle$ for all $g\in H$. In particular $\pi$ extends $\pi_H$. 
\end{proof}

\begin{lem}\label{klinearcontext}
Let $V_1,\dots,V_k$ be locally convex spaces and $W$ be a Banach space. Further let $L_1\subset V_1,\dots,L_k\subset V_k$ be dense subspaces and $m': L_1 \times \dots \times L_k \rightarrow W$ be a continuous $k$-linear map. Then $m'$ is uniquely extendable to a continuous $k$-linear map $m: V_1 \times \dots \times V_k \rightarrow W$.
\end{lem}
\begin{proof}
Choose non-empty convex balanced open subsets $U_i\subset V_i$ for all $i\in \{1,\dots,k\}$ such that $m'$ is bounded on $U_L:=U\cap (L_1 \times \dots \times L_k)$ where $U:=U_1\times \cdots \times U_k$. It is easily verified that $m'$ is uniformly continuous on each $n\cdot U_L$. Since $n\cdot U_L$ is dense in $n\cdot U$ we obtain a $k$-linear continuous extension $m_n: n\cdot U  \rightarrow W$ of $m'\vert_{n\cdot U_L}$ for every $n\in \N$ (cf. \cite[II.3.6 Thm. 2]{Bo66}). Since $\bigcup_{n\in \N} n\cdot U=V_1\times \dots \times V_k$, these extensions fit together to a $k$-linear continuous extension $m: V_1 \times \dots \times V_k \rightarrow W$ of $m'$.
\end{proof}

Before turning to dense locally exponential subgroups, we state a slightly more general theorem on extending smooth resp. semibounded representations.

\begin{thm}\label{smoothsemibextensionmoregen}
Let $G, H$ be locally convex Lie groups with exponential maps and Lie algebras $\g, \h$. Let $f:H\rightarrow G$ be a morphism of Lie groups and $\pi:G\rightarrow \U(\H)$ a continuous unitary representation. Assume that $\L(f):\h\hookrightarrow \g$ is a topological embedding with dense range. Then the following assertions hold.
\begin{itemize}
\item[(a)] If $G$ is locally exponential and $\pi\circ f$ is smooth then $\pi$ is smooth. Moreover, the space of smooth vectors for $\pi$ and $\pi\circ f$ coincide.
\item[(b)] If $\pi$ is smooth and $\pi\circ f$ is semibounded then $\pi$ is semibounded.
\end{itemize}
\end{thm}
\begin{proof}
(a) 
We set $\pi_H:=\pi\circ f$, $\h_1:=\L(f)(\h)$ and denote the set of smooth vectors for $\pi_H$ by $\H^\infty(\pi_H)$ and for $\pi$ by $\H^\infty(\pi)$. In this proof we denote by $\dd\pi(x)$ the skew-adjoint generator of the unitary one-parameter group $t\mapsto \pi(\exp(tx))$ for all $x\in \g$. Note that $$\overline{\dd\pi_H(x)\vert_{\H^\infty(\pi_H)}}=\dd\pi(\L(f)(x))$$ holds for $x\in\h$. For every $v\in \H^\infty(\pi_H)$ and $n\in \N$, the continuous $n$-linear map
$$\h_1 \times \dots \times \h_1 \rightarrow \H,(x_1,\dots,x_n) \mapsto \dd\pi_H(\L(f)^{-1}x_1)\cdots\dd\pi_H(\L(f)^{-1}x_n)v$$
is uniquely extendable to a continuous map $m_n^v:\g \times \dots \times \g \rightarrow \H$ by Lemma \ref{klinearcontext}, since $\h_1\subset \g$ is dense. We set $m_0^v:=v$ for $v\in \H^\infty(\pi_H)$. For $n\in \N$, we consider the subspaces $\tilde{\D}_n := \bigcap_{x_1,\dots,x_n\in \g} \D(\dd\pi(x_1) \cdots \dd\pi(x_n))$ and 
$$\D_n:= \left\{ v\in \tilde{\D}_n : \dd\pi(x_1)\cdots\dd\pi(x_n)v=m_n^v(x_1,\dots,x_n) \text{ for all $x_i\in \g$} \right\}$$
of $\H$.\newline
\textbf{Step 1:} For $n\in \N_0, v\in \H^\infty(\pi_H),t\in \R^\times$ and $x_0,x_1,\dots,x_n\in \g$ we have:
\begin{align}\label{smoothextensiongeneq1}
\pi(\exp(tx_0))m_n^v(x_1,\dots,x_n)-m_n^v(x_1,\dots,x_n) &= t\int_0^1\pi(\exp(tsx_0))m_{n+1}^v(x_0,\dots,x_n)\ ds.
\end{align}
This can be seen as follows. First let $x_0,x_1,\dots,x_n\in \h_1$, choose $y_0\in \h$ such that $x_0=\L(f)(y_0)$ and set $\tilde{v}:=\dd\pi_H(\L(f)^{-1}x_1)\cdots\dd\pi_H(\L(f)^{-1}x_n)v$. Note that $\tilde{v}\in \H^\infty(\pi_H)$. The Fundamental Theorem of Calculus applied to the smooth curve $c(s)=\pi(\exp(tsx_0))\tilde{v}=\pi_H(\exp(tsy_0))\tilde{v}$ yields
$$\pi(\exp(tx_0))\tilde{v}-\tilde{v} = t\int_0^1\pi(\exp(tsx_0))\dd\pi_H(\L(f)^{-1}x_0)\tilde{v}\ ds,$$
which shows that \eqref{smoothextensiongeneq1} holds for all $x_0,\dots,x_n\in \h_1$. The map 
$$\R^2 \times \g \times \dots \times \g \rightarrow \H,(t,s,x_0,\dots,x_n) \mapsto \pi(\exp(tsx_0))m_{n+1}^v(x_0,\dots,x_n)$$
is continuous since $m_{n+1}^v$ and the action $G\times \H\rightarrow \H,(g,w)\mapsto \pi(g)w$ are continuous. Hence the right hand side of \eqref{smoothextensiongeneq1} is continuous in $(x_0,\dots,x_n)\in \g^{n+1}$. Since the left hand side of \eqref{smoothextensiongeneq1} is also continuous in $(x_0,\dots,x_n)\in \g^{n+1}$ and since $\h_1\subset \g$ is dense we conclude that \eqref{smoothextensiongeneq1} holds for all $(x_0,\dots,x_n)\in \g^{n+1}$.\newline
\textbf{Step 2:} $\H^\infty(\pi_H) \subset \D_k$ for all $k\in \N$.\newline
We show this by induction on $k\in \N$. Let $v\in\H^\infty(\pi_H)$ and consider equation \eqref{smoothextensiongeneq1} with $n=0$. Dividing by $t$ and letting $t\rightarrow 0$ in this equation we obtain $v \in \D(\dd\pi(x))$ with $\dd\pi(x)v=m_1^v(x)$ for all $x\in\g$. Hence $v\in \D_1$. This shows $\H^\infty(\pi_H)\subset \D_1$. Now assume that $\H^\infty(\pi_H) \subset \D_k$ holds for some $k\in \N$. Let $v\in \H^\infty(\pi_H)$ and $x_0,x_1,\dots,x_k\in \g$. Then by assumption $\tilde{v}:=\dd\pi(x_1)\dots\dd\pi(x_k)v=m_k^v(x_1,\dots,x_k)$. Now consider equation \eqref{smoothextensiongeneq1} with $n=k$. Dividing by $t$ and letting $t\rightarrow 0$ in this equation we obtain $\tilde{v} \in \D(\dd\pi(x_0))$ with $\dd\pi(x_0)\tilde{v}=m_{k+1}^v(x_0,\dots,x_k)$ for all $x\in\g$. We conclude $v\in\D_{k+1}$ and thus $\H^\infty(\pi_H) \subset \D_{k+1}$. Hence Step 2 is proven.

Since $G$ is locally exponential we know from \cite[Lem.3.4]{Ne10a} that $\bigcap_k \D_k = \H^\infty(\pi)$. So Step 2 shows that $\H^\infty(\pi_H)=\H^\infty(\pi)$ and in particular $\pi$ is smooth.\newline
(b) Since $\pi\circ f$ is semibounded, we find a non-empty open subset $U\subset \h$ such that the support function of $\pi\circ f$ is bounded on $U$, i.e., 
$$\sup_{x\in U} s_{\pi\circ f}(x) \leq c <\infty$$
for some $c\in \R$. As $s_{\pi\circ f}=s_\pi \circ \L(f)$ and since $s_\pi$ is lower semicontinuous we obtain
$$\sup_{x\in \overline{\L(f)U}} s_{\pi}(x) \leq c <\infty.$$
Since $\L(f):\h\rightarrow \g$ is a topological embedding with dense image, the subset $\overline{\L(f)U}\subset \g$ contains interior points. Hence $\pi$ is semibounded.
\end{proof}

\begin{dfn}
Let $G$ be a locally exponential Lie group. A (not necessarily closed) subgroup $H$ of $G$ is called a \textit{locally exponential Lie subgroup}, if $H$ carries the structure of a locally exponential Lie group compatible with the subspace topology $H\subset G$. By \cite[Remark~IV.1.22]{Ne06}, this structure is unique if it exists. If $H\subset G$ is a locally exponential Lie subgroup, then the inclusion $H\subset G$ is smooth and it induces a topological embedding of the corresponding Lie algebras $\h \subset \g$. We regard $\h$ as a subspace of $\g$ in this case.
\end{dfn}

\begin{cor}\label{smoothextensiongen}
Let $G$ be a locally exponential Lie group and $H\subset G$ be a dense locally exponential Lie subgroup. Let $\pi_H:H\rightarrow \U(\H)$ be a smooth representation. Then $\pi_H$ extends uniquely to a continuous representation $\pi:G\rightarrow \U(\H)$. This representation $\pi$ is smooth and the space of smooth vectors for $\pi$ and $\pi_H$ coincide.
\end{cor}
\begin{proof}
By Proposition \ref{contextension} the representation $\pi_H$ extends uniquely to a continuous representation $\pi:G\rightarrow \U(\H)$. Let $f:H\hookrightarrow G$ denote the inclusion. Then $\pi_H=\pi\circ f$ and the assertion follows from Theorem~\ref{smoothsemibextensionmoregen}(a).
\end{proof}

\begin{cor}\label{smoothextension}
Let $L\subset V_A=C^\infty(A)$ be a dense real subspace which is invariant under both $D$ and $\gamma$. Consider the oscillator group $G_L:=\Heis(L,\omega_A) \rtimes_\gamma \R$ with Lie algebra $\g_L$, where we assume $\exp_{G_A}(\g_L)\subset G_L$. Let $\pi_L:G_L \rightarrow \U(\H)$ be a smooth representation. Then the following assertions hold:
\begin{itemize}
\item[\rm(a)] $\pi_L$ is uniquely extendable to a smooth representation $\pi: G_A \rightarrow \U(\H)$. 
\item[\rm(b)] If $D(L) \subset L$ is dense and $B(I_{\pi_L})^0\supset W_d\cap \g_L$ holds for some $d\in\overline{\R}$, then $\pi$ is semibounded with $B(I_\pi)^0 \supset W_d$.
\end{itemize}
\end{cor}
\begin{proof}
By Proposition \ref{contextension} there exists a unique continuous extension $\pi: G_A \rightarrow \U(\H)$ of $\pi_L$. Let $v\in \H$ be a smooth vector for $\pi_L$. Since $\Heis(V_A)$ and $\Heis(L)$ are locally exponential we obtain by Corollary \ref{smoothextensiongen} that $v$ is smooth for $\pi\vert_{\Heis(V_A)}$. But then the map
$$G_A \rightarrow \C, (t,x,s) \mapsto \langle \pi(t,x,s)v,v \rangle = \langle\pi_L(0,0,s)v,\pi(-t,-x)v \rangle$$
is smooth and thus $v$ is smooth for $\pi$ (cf. \cite[Thm.7.2]{Ne10a}). It follows that $\pi$ is smooth. Now suppose that  $D(L) \subset L$ is dense and $W_d\cap \g_L \subset B(I_{\pi_L})^0$. Let $\hat s:W_d \rightarrow \R$ be the continuous extension of $s_{\pi_L}\vert_{W_d\cap \g_L}$ from Remark \ref{supportfunctiongenrmk}. Recall 
$$s_{\pi_L}(a)=\sup_{v\in \H^\infty,\|v\|=1} \langle i\dd\pi_L(a)v,v\rangle.$$
In particular $\langle i\dd\pi(a)v,v\rangle \leq \hat s(a)$ for all $a\in W_d\cap \g_L, \|v\|=1$. Since $W_d\cap\g_L$ is dense in $W_d$, we obtain by continuity 
$$s_\pi(a) = \sup_{v\in\H^\infty,\|v\|=1} \langle i\dd\pi(a)v,v\rangle \leq \hat{s}(a)$$
for all $a\in W_d$. Therefore $\pi$ is semibounded with $W_d\subset B(I_\pi)^0$.
\end{proof}

\section{Application to semibounded representations of $G_A$}\label{applsbrepr}

With the results from Section \ref{extensionsofrepofgenliegroups} we show that every semibounded representation $\pi:G_A\rightarrow\U(\H)$ extends to a holomorphic representation of an associated complex semigroup. This result together with the results from \cite{Ne08} allow us to apply $C^*$-methods to semibounded representations of $G_A$, which will be discussed in Section \ref{sectc^*algebras} below.

\subsection{Holomorphic extensions of semibounded representations of $G_A$}\label{sectholextsbrgapart2}

Recall from Section \ref{sectdefsd} the complex involutive semigroups $S_d=G_A^\cO\exp_\C(iW_d^\cO)\subset G_{A,\C}$ for $d\in \overline{\R}$. Let us consider another topology. 
\begin{dfn}
We denote by $\tau_V$ resp. $\tau_{V^\cO}$ the topology on $V=C^\infty(A)$ resp. on $V^\cO$ given by the norm $x\mapsto \|x\|+\|Ax\|$. Let $\tau_X$ denote the topology on $X\in\{G_{A,\C}, G_A^\cO, S_d\}$ induced by the topology $\tau_{V^\cO}$ on $V^\cO$.
\end{dfn}

\begin{rmk}\label{tautoprmk}
Note that the groups $(X,\tau_X)$ from the preceding definition are not Lie groups resp. Lie semigroups, but we may consider $(X,\tau_X)$ as a locally convex space for $X\in\{G_{A,\C}, G_A^\cO\}$. Note also that $\Heis(V,\omega_A)$ is a Lie group when $V$ is equipped with $\tau_V$. The topological space $(S_d,\tau_{S_d})$ may be considered as an open subset in the locally convex space $(G_{A,\C},\tau_{G_{A,\C}})$ since $S_d$ is open in $(G_{A,\C},\tau_{G_{A,\C}})$. This may be seen as follows: Let $\vp: S_\infty \rightarrow G_A^\cO \times W_\infty^\cO$ be the inverse of the polar map. By Remark \ref{openinvconesgenrmk}, the cone $W_d^\cO\subset \g_A^\cO$ is open when $\g_A^\cO$ is equipped with the topology induced by the $\|\cdot\|$-topology on $V^\cO$. Hence Proposition \ref{polarmap}(c) yields that $S_d=\vp^{-1}(G_A^\cO\times W_d^\cO)$ is open in $(G_{A,\C},\tau_{G_{A,\C}})$.
\end{rmk}

As $(G_A^\cO,W_d^\cO)$ is complexifiable by Proposition \ref{G_Acomplexifiable}, Theorem \ref{sbrepholrepcorrespondencegen} yields a one-to-one correspondence between semibounded representations $\nu: G_A^\cO \rightarrow \U(\H)$ with $W_d^\cO\subset B(I_\nu)$ and non-degenerate holomorphic representations $\rho:S_d\rightarrow B(\H)$. However, it does not say if $\nu$ is smooth when $G_A^\cO\subset G_A$ is equipped with the subspace topology nor does it say whether $\nu$ may be extended to $G_A$. This will be part of the next theorem.

Let $\pi:G_A\rightarrow \U(\H)$ be a semibounded representation. From Proposition \ref{openinvconesgen} we know that $B(I_\pi)^0$ is one of the cones $\pm W_d, d\in \overline{\R}$ or $\g$. Since we may switch from $B(I_\pi)^0$ to $-B(I_\pi)^0$, we may assume $W_d\subset B(I_\pi)^0$ for some $d\in\overline{\R}$. Note that $\pi_h:=\pi\vert_{G_A^\cO}$ is a semibounded representation since the topology on $G_A^\cO$ is finer than that on $G_A$. Note also $W_d^\cO\subset B(I_{\pi_h})^0$. Since $(G_A^\cO,W_d^\cO)$ is complexifiable by Proposition \ref{G_Acomplexifiable}, we obtain by Theorem \ref{holextgen} the holomorphic representation $\hat \pi_h:S_d\rightarrow B(\H)$ of the complex involutive semigroup $S_d$. We set $\hat\pi:=\hat \pi_h$. Note that $\hat\pi$ also depends on $d$.

\begin{thm}\label{sbrepholrepcorrespondenceG_A}
Let $d\in \overline{\R}$. The assignment $\pi\mapsto \hat\pi$ yields a bijection
$$\Phi:\left\{
\begin{array}[c]{c}
\text{$\pi:G_A\rightarrow \U(\H)$ semibounded} \\ 
\text{representations with $W_d\subset B(I_\pi)^0$}
\end{array}
\right\}
\rightarrow
\left\{
\begin{array}[c]{c}
\text{$\rho:S_d\rightarrow B(\H)$ non-degenerate}\\
\text{holomorphic representations}
\end{array}
\right\}$$
which preserves commutants, i.e., $\pi(G_A)'=\hat\pi(S_d)'$. Moreover, every $\Phi(\pi)$ is holomorphic as a map $\Phi(\pi):(S_d, \tau_{S_d})\rightarrow B(\H)$ and for every $\pi$ the representation $\pi\vert_{\Heis(V)}$ is smooth when $V$ is equipped with topology $\tau_V$.
\end{thm}
\begin{proof}
Let $d\in \overline{\R}$. Suppose $\Phi(\pi_1)=\Phi(\pi_2)$. Then $\pi_1\vert_{G_A^\cO}=\pi_2\vert_{G_A^\cO}$ by Theorem \ref{sbrepholrepcorrespondencegen}. Since $G_A^\cO$ is dense in $G_A$ and $\pi_1, \pi_2$ are strongly continuous we conclude $\pi_1=\pi_2$. Hence $\Phi$ is injective. Now let $\rho:S_d\rightarrow B(\H)$ be a non-degenerate holomorphic representation. By Theorem \ref{sbrepholrepcorrespondencegen}, there is a semibounded representation $\pi_h:G_A^\cO \rightarrow B(\H)$ with $\hat \pi_h = \rho$ and $W_d^\cO \subset B(I_{\pi_h})^0$. From Proposition \ref{supportfunctiongen}, applied to the oscillator groups $G_A^\cO\subset G_A$ we know that $s_{\pi_h}\vert_{W_d^\cO}$ is continuous when $W_d^\cO$ is equipped with the topology induced by the $\|\cdot\|$-topology on $V^\cO$. By Proposition \ref{polarmap}, the inverse of the polar map $$\vp_d:S_d\rightarrow G_A^\cO\times W_d^\cO, g\exp_\C(iw)\mapsto (g,w)$$
is continuous when $G_A^\cO \times W_d^\cO$ is equipped with the topology induced by the $\|\cdot\|$-topology on $V^\cO$ and $S_d$ is equipped with $\tau_{S_d}$. Let $P:G_A^\cO\times W_d^\cO\rightarrow W_d^\cO,(g,w)\mapsto w$. For $s=g\exp_\C(iw)\in S_d$ we then have
$$\|\hat\pi_h(s)\|=e^{s_{\pi_h}(w)}=e^{s_{\pi_h}(P\vp_d(s))}.$$
Hence $\hat\pi_h$ is locally bounded when $S_d$ is equipped with $\tau_{S_d}$. Since $\hat\pi_h$ is holomorphic on affine discs (of complex dimension $1$) in $S_d$, we obtain that $\hat\pi_h$ is holomorphic w.r.t. $\tau_{S_d}$, c.f. \cite[Thm. 6.2]{BS71} and Remark \ref{tautoprmk}. With the explicit formula for the multiplication in $G_{A,\C}$ it is easily verified that, for every $s\in S_d$, the map
$$G_A^\cO\rightarrow S_d,g \mapsto gs$$
is smooth when $G_A^\cO$ is equipped with $\tau_{G_A^\cO}$ and $S_d$ is equipped with $\tau_{S_d}$. Thus, for every $s\in S_d$ and $v\in \H$, the map 
$$G_A^\cO \rightarrow \H,g\mapsto \pi_h(g)\hat\pi_h(s)v=\hat\pi_h(gs)v$$
is smooth when $G_A^\cO$ is equipped with the topology $\tau_{G_A^\cO}$. Since $\spn(\hat\pi_h(S_d)\H)$ is dense in $\H$, we obtain that $\pi_h\vert_{\Heis(V^\cO)}$ is smooth w.r.t. $\tau_{V^\cO}$ and that $\pi_h$ is smooth when $G_A^\cO\subset G_A $ is equipped with the subspace topology. By Corollary \ref{smoothextension}(a) and Corollary \ref{smoothextensiongen} we therefore obtain a smooth representation $\pi:G_A\rightarrow \U(\H)$ extending $\pi_h$ such that $\pi\vert_{\Heis(V)}$ is smooth when $V$ is equipped with the topology $\tau_V$. Since $W_d^\cO=W_d\cap \g_A^\cO \subset B(I_{\pi_h})^0$ and since $D(V^\cO)\subset V$ is dense, we obtain from Corollary \ref{smoothextension}(b) that $\pi$ is semibounded with $W_d\subset B(I_\pi)^0$. By construction we have $\Phi(\pi)=\hat\pi_h= \rho$. Thus $\Phi$ is surjective. Since $G_A^\cO\subset G_A$ is dense and $\pi$ is strongly continuous we have $\pi(G_A)'=\pi(G_A^\cO)'=\hat\pi(S_d)'$, where the last equation holds by Theorem \ref{sbrepholrepcorrespondencegen}.
\end{proof}

\begin{cor}\label{semibsmoothtau_V}
Let $\pi:G_A\rightarrow \U(\H)$ be a semibounded representation. Then $\pi\vert_{\Heis(V)}$ is a smooth representation when $V$ is equipped with the topology $\tau_V$.
\end{cor}
\begin{proof}
Recall that we may assume $B(I_\pi)^0\subset W_\infty$ and therefore $B(I_\pi)^0=W_d$ for some $d\in \overline{\R}$. Hence the assertion follows from Theorem \ref{sbrepholrepcorrespondenceG_A}.
\end{proof}

\subsection{$C^*$-algebras associated to $G_A$}\label{sectc^*algebras}
In the preceding subsection we showed that, for $d\in\overline\R$, the semibounded representations $\pi:G_A\rightarrow \U(H)$ satisfying $W_d\subset B(I_\pi)^0$ are in one-to-one correspondence with non-degenerate holomorphic representations $\rho:S_d\rightarrow B(\H)$ (Theorem \ref{sbrepholrepcorrespondenceG_A}). From \cite{Ne08} it is known that, for every locally bounded absolute value $\alpha$ on $S_d$, there is a $C^*$-algebra $C^*(S_d,\alpha)$ such that the $\alpha$-bounded non-degenerate holomorphic representations of $S_d$ correspond to the representations of $C^*(S_d,\alpha)$. Combining these results will allow us to apply $C^*$-methods to semibounded representations of $G_A$. In particular, we show that for a separable oscillator group $G_A$ every semibounded representation of $G_A$ on a separable Hilbert space is a direct integral of semibounded factor representations.

Let $S$ be a complex involutive semigroup and $\alpha$ be a locally bounded absolute value on $S$.

\begin{dfn}
To the pair $(S,\alpha)$ we associate the $C^*$-algebra $C^*(S,\alpha)$ as defined in \cite[Def. 3.3]{Ne08}. According to \cite[Thm. 3.5]{Ne08}, there exists a holomorphic morphism $\eta_\alpha:S\rightarrow C^*(S,\alpha)$ of involutive semigroups with total range, i.e., $\eta_\alpha(S)$ spans a dense subspace of $C^*(S,\alpha)$, such that $\|\eta_\alpha(s)\| \leq \alpha(s), s\in S$, and the following holds:
\end{dfn}

\begin{prop}\label{semigrprepc*repcorrespondence}
For every $\alpha$-bounded holomorphic representation $\pi:S\rightarrow B(\H)$ there exists a unique representation $\tilde{\pi}:C^*(S,\alpha)\rightarrow B(\H)$ such that $\pi=\tilde{\pi}\circ \eta_\alpha$ holds. Conversely, for every representation $\tilde{\pi}:C^*(S,\alpha)\rightarrow B(\H)$, the representation $\tilde{\pi}\circ \eta_\alpha\colon S\rightarrow B(\H)$ is holomorphic and $\alpha$-bounded. The commutants of $\tilde{\pi}$ and $\tilde{\pi}\circ \eta_\alpha$ coincide, i.e. $\tilde{\pi}(C^*(S,\alpha))'=\tilde{\pi}(\eta_\alpha(S))'$.
\end{prop}
\begin{proof}
The first statement is \cite[Thm. 3.5(ii)]{Ne08}. The second statement holds, since $\tilde{\pi}$ is holomorphic and contractive as a morphism of $C^*$-algebras and since $\|\eta_\alpha(s)\|\leq \alpha(s)$ for all $s\in S$. The equality of commutants $\tilde{\pi}(C^*(S,\alpha))'=\tilde{\pi}(\eta_\alpha(S))'$ is easily derived from the fact that $\eta_\alpha(S)$ is total in $C^*(S,\alpha)$.
\end{proof}

\begin{rmk}
\begin{itemize}
\item[\rm(a)] A representation $\tilde{\pi}:C^*(S,\alpha)\rightarrow B(\H)$ is non-degenerate if and only if the holomorphic representation $\tilde{\pi} \circ \eta_\alpha:S \rightarrow B(\H)$ is non-degenerate. This follows from the fact that $\eta_\alpha(S)$ is total in $C^*(S,\alpha)$.
\item[\rm(b)] $C^*(S,\alpha)$ is separable if $S$ is separable, since $\eta_\alpha(S)$ is total in $C^*(S,\alpha)$. If $S=S_d=G_A^\cO\exp_\C(iW_d^\cO)$ for some $d\in \overline{\R}$, then $S$ is separable if and only if $V=C^\infty(A)\subset H_A$ is separable, which is equivalent to the separability of $H_A$.
\end{itemize}
\end{rmk}

The preceding proposition entails that the $\alpha$-bounded holomorphic representations of $S$ are in one-to-one correspondence with the representations of the $C^*$-algebra $C^*(S,\alpha)$. For separable $C^*$-algebras there is a good disintegration theory. We shall show now that the correspondence in Theorem \ref{sbrepholrepcorrespondenceG_A} behaves well under disintegration. For the concept of direct integrals we refer to \cite[Part II]{Di81}.

\begin{lem}\label{dirintlem}
Assume that $V=C^\infty(A)$ is separable. Let $\alpha:S_d \rightarrow \R$ be a locally bounded absolute value and $\eta_\alpha:S_d\rightarrow C^*(S_d,\alpha)$ as above. Let $\tilde{\pi}:C^*(S_d,\alpha)\rightarrow B(\H)$ be a non-degenerate representation. Suppose that $\H$ is separable and that
$$(\tilde{\pi}, \H) \cong \left(\int^\oplus_X \tilde{\pi}_\lambda \ d \mu(\lambda), \int^\oplus_X \H_\lambda \ d\mu(\lambda)\right),$$
where $X$ is a Borel space, $\mu$ a positive measure on $X$ and $\lambda \mapsto \tilde{\pi}_\lambda$ a $\mu$-measurable field of representations of $C^*(S_d,\alpha)$ on the $\H_\lambda$.  Then the following assertions hold.
\begin{itemize}
\item[\rm(a)] There exists a $\mu$-zero set $N\subset X$ such that $\tilde{\pi}_\lambda$ is non-degenerate for every $\lambda \in X \backslash N$.
\item[\rm(b)] Let $\pi:=\Phi^{-1}(\tilde{\pi}\circ\eta_\alpha)$, where $\Phi$ is as in Theorem \ref{sbrepholrepcorrespondenceG_A}.  Set $\pi_\lambda:=\Phi^{-1}(\tilde{\pi}_\lambda\circ \eta_\alpha)$ for all $\lambda \in X \backslash N$ and let $\pi_\lambda$ be the trivial representation for $\lambda \in N$. Then we have
$$(\pi, \H) \cong \left(\int^\oplus_X \pi_\lambda \ d \mu(\lambda), \int^\oplus_X \H_\lambda \ d\mu(\lambda)\right).$$
\end{itemize}
\end{lem}
\begin{proof}
Part (a) follows from \cite[8.1.5]{Di77}. We set $G:=G_A^\cO$ and $S:=S_d$. Let $X':=X\backslash N$ and let us identify $(\tilde{\pi}, \H)$ with $(\int^\oplus_X \tilde{\pi}_\lambda \ d \mu(\lambda),\int^\oplus_X \H_\lambda \ d\mu(\lambda))$. Recall 
\begin{align}\label{dirintlemeq1}
\pi(g)\tilde{\pi}(\eta_\alpha(s))=\tilde{\pi}(\eta_\alpha(gs)) \quad \text{ and } \quad \pi_\lambda(g)\tilde{\pi}_\lambda(\eta_\alpha(s))=\tilde{\pi}_\lambda(\eta_\alpha(gs))
\end{align}
for all $g\in G, s\in S, \lambda \in X'$ by Remark \ref{pihatrmks}(a). Thus the field $\lambda \mapsto \pi_\lambda(g)\tilde{\pi}_\lambda(\eta_\alpha(s))v_\lambda$ on $X$ is $\mu$-measurable for all $g\in G, s\in S, v=\int^\oplus v_\lambda \ d\mu(\lambda)\in \H$. This implies that $\lambda \mapsto \pi_\lambda(g)v_\lambda$ is $\mu$-measurable for every $v=\int^\oplus v_\lambda \ d\mu(\lambda)\in \H^0:=\spn(\tilde{\pi}(\eta_\alpha(S))\H)$. Now let $v=\int^\oplus v_\lambda \ d\mu(\lambda)\in \H$. Since $\H^0$ is dense in $\H$, we find a sequence $v_n=\int^\oplus v_{\lambda,n} \ d\mu(\lambda)$ in $\H^0$ such that $v_{\lambda,n}\rightarrow v_\lambda$ holds for $\mu$-almost all $\lambda$ in $X$. Thus, for every $g\in G$, the field $\lambda \mapsto \pi_\lambda(g)v_\lambda$ is $\mu$-measurable as a pointwise-limit of a sequence of $\mu$-measurable fields. Moreover, since $G\subset G_A$ is dense, we find for each $g\in G_A$ a sequence $g_n \in G$ with $g_n\rightarrow g$. Since $\pi_\lambda$ is strongly continuous, we now see that for every $g\in G_A$ the field $\lambda \mapsto \pi_\lambda(g)v_\lambda$ is $\mu$-measurable as a pointwise-limit of a sequence of $\mu$-measurable fields. Hence we have shown that $X \ni \lambda \mapsto \pi_\lambda(g)$ is a $\mu$-measurable field of operators for every $g\in G_A$ and we therefore may consider the decomposable operator $\int^\oplus \pi_\lambda(g)\ d\mu(\lambda)\in \U(\H)$ for $g\in G_A$. From the equations \eqref{dirintlemeq1} we easily derive that the operators $\pi(g)$ and  $\int^\oplus \pi_\lambda(g)\ d\mu(\lambda)$ coincide on the dense subspace $\H^0$ for every $g\in G$. By continuity we therefore have $\pi(g) = \int^\oplus \pi_\lambda(g)\ d\mu(\lambda)$ for all $g\in G_A$.
\end{proof}

\begin{cor}[Disintegration into factor representations]
Suppose that $C^\infty(A)$ and $\H$ are separable and let $\pi:G_A\rightarrow \U(\H)$ be a semibounded representation. Let $W_d\subset B(I_\pi)^0$ for some $d\in \overline{\R}$. Then we have
$$(\pi,\H) \cong \left(\int^\oplus_X \pi_\lambda \ d \mu(\lambda), \int^\oplus_X \H_\lambda \ d\mu(\lambda)\right),$$
where $X$ is a standard Borel space, $\mu$ a positive measure on $X$ and $\pi_\lambda:G_A\rightarrow \U(\H_\lambda)$ is a semibounded factor representation with $W_d\subset B(I_\pi)^0$ for all $\lambda \in X$.
\end{cor}
\begin{proof}
Consider the non-degenerate holomorphic representation $\hat \pi :S_d\rightarrow B(\H)$ from Theorem~\ref{sbrepholrepcorrespondenceG_A}. We set $\alpha(s):=\|\hat \pi(s)\|$ and consider the associated non-degenerate representation $\tilde{\pi}:C^*(S_d,\alpha)\rightarrow B(\H)$ satisfying $\hat \pi = \tilde{\pi}\circ \eta_\alpha$ from Proposition \ref{semigrprepc*repcorrespondence}. By applying \cite[Thm. 8.4.2]{Di77}, we obtain
$$(\tilde{\pi},\H) \cong \left(\int^\oplus_X \tilde{\pi}_\lambda \ d \mu(\lambda), \int^\oplus_X \H_\lambda \ d\mu(\lambda)\right),$$
where $X$ is a standard Borel space, $\mu$ a positive measure on $X$ and $\tilde{\pi}_\lambda$ is a factor representation for every $\lambda \in X$. Lemma \ref{dirintlem} yields
$$(\pi,\H) \cong \left(\int^\oplus_X \pi_\lambda \ d \mu(\lambda), \int^\oplus_X \H_\lambda \ d\mu(\lambda)\right).$$
Here $\pi_\lambda=\Phi^{-1}(\tilde{\pi}_\lambda\circ \eta_\alpha)$ for $\lambda \in X \backslash N$ and $\pi_\lambda$ is the trivial representation for $\lambda \in N$, where $N\subset X$ is as in Lemma \ref{dirintlem}. In particular $\pi_\lambda(G_A)'=\tilde{\pi}_\lambda(C^*(S_d,\alpha))'$ for all $\lambda \in X\backslash N$, cf. Theorem \ref{sbrepholrepcorrespondenceG_A} and Proposition \ref{semigrprepc*repcorrespondence}. Hence $\pi_\lambda$ is a semibounded factor representation for all $\lambda \in X$.
\end{proof}

\begin{cor}[Central disintegration]\label{centraldisint}
Suppose that $C^\infty(A)$ and $\H$ are separable and let $\pi:G_A\rightarrow \U(\H)$ be a semibounded representation.  Then we have
$$(\pi,\H) \cong \left(\int^\oplus_{\R} \pi_\lambda \ d \mu(\lambda), \int^\oplus_\R \H_\lambda \ d\mu(\lambda)\right),$$
where $\mu$ a positive measure on $\R$ and $\pi_\lambda:G_A\rightarrow \U(\H_\lambda)$ is a semibounded representation with $\pi_\lambda(t,0,0)=e^{i\lambda t}\1$ for all $t\in\R$ and almost all $\lambda\in \R$.
\end{cor}
\begin{proof}
Assume w.l.o.g. $W_d\subset B(I_\pi)^0$ for some $d\in \overline{\R}$. Consider the holomorphic extension $\hat \pi :S_d\rightarrow B(\H)$ from Theorem \ref{sbrepholrepcorrespondenceG_A}, set $\alpha(s):=\|\hat \pi(s)\|$ and recall the associated non-degenerate representation $\tilde{\pi}:C^*(S_d,\alpha)\rightarrow B(\H)$ satisfying $\hat \pi = \tilde{\pi}\circ \eta_\alpha$ from Proposition \ref{semigrprepc*repcorrespondence}. Let $\pi_c(t):=\pi(t,0,0)$ for $t\in \R$, so that $\pi_c:\R \rightarrow \U(\H)$ is a continuous unitary representation. The Spectral Theorem yields
$$(\pi_c,\H) \cong \left(\int^\oplus_\R \pi_{c,\lambda} \ d\mu(\lambda), \int^\oplus \H_\lambda \ d\mu(\lambda)\right),$$
where $\mu$ is a positive measure on $\R$ and $\pi_{c,\lambda}(t)=e^{i\lambda t}\1$, cf. e.g. \cite[sect. I.4.A.1]{GV64}. Let us consider this equivalence as an identification of the corresponding spaces. One may verify that $\pi_c(\R)''$ is the algebra of diagonalizable operators on $\H=\int^\oplus \H_\lambda \ d\mu(\lambda)$. Since $\R\times \{0\} \times \{0\}$ is the center of $G_A$ we have $\tilde{\pi}(C^*(S_d,\alpha))''=\pi(G_A)'' \subset \pi_c(\R)'$. Thus every $\tilde{\pi}(a), a\in C^*(S_d,\alpha)$, is a decomposable operator, cf. \cite[II.2.5. Cor.]{Di81}. Therefore \cite[Lem. 8.3.1]{Di77} yields a direct integral decomposition
$$(\tilde{\pi},\H) \cong \left(\int^\oplus_{\R} \tilde{\pi}_\lambda \ d \mu(\lambda), \int^\oplus_\R \H_\lambda \ d\mu(\lambda)\right).$$
With Lemma \ref{dirintlem} we obtain further
$$(\pi,\H) \cong \left(\int^\oplus_\R \pi_\lambda \ d \mu(\lambda), \int^\oplus_\R \H_\lambda \ d\mu(\lambda)\right),$$
where $\pi_\lambda:G_A\rightarrow \U(\H)$ are semibounded representations for all $\lambda \in \R$. In particular $\int^\oplus \pi_\lambda(t,0,0) \ d\mu(\lambda)=\pi(t,0,0)=\pi_c(t)=\int^\oplus e^{i\lambda t}\1 \ d\mu(\lambda)$. Hence there is a $\mu$-zero set $N\subset \R$ such that $\pi_\lambda(t,0,0)=e^{i\lambda t}\1$ holds for all $\lambda \in \R\backslash N, t\in \Q$. By continuity this also holds for all $t\in \R$. Now the statement follows.
\end{proof}

\begin{rmk}
If $\inf\Spec(A)>0$ we show in \cite{Ze15} that the $C^*$-algebras $C^*(S_d,\alpha)$ are postliminal and all semibounded representations of $G_A$ are of type I.
\end{rmk}

\subsection{A smoothness condition}

The main result of this subsection is Theorem \ref{extensionofdensesubgrepr}. It gives a criterion for the extendability of representations of certain dense direct limit subgroups of $G_A$, which are assumed to be smooth w.r.t. the finer direct limit topology, to semibounded representations of $G_A$. This plays an important role in the study of semibounded representations of $G_A$ (cf. \cite{Ze15}).

\begin{rmk}\label{complexlimitrmk}
Consider a standard oscillator group $G_A=\Heis(V,\omega_A)\rtimes_\gamma \R$, $V=C^\infty(A)$. Suppose we are given $\gamma$-invariant complex subspaces $V_n\subset V$, which are closed in $V$, with $V_n\subset V_{n+1}$ for all $n\in\N$. Further we assume that $V_{(\infty)}:=\bigcup_n V_n$ is dense in $V$. Then $\gamma$ restricts to unitary one-parameter groups $\R \rightarrow \U(V_n)$ for all $n\in\N$ with self-adjoint generator $A_n=\overline{A\vert_{V_n}}$ with domain $\D(A_n)\subset H_{A_n}$, where the Hilbert completion $H_{A_n}$ of $V_n$ may be considered as a closed subspace of $H_A$. Note that $C^\infty(A_n)=V_n$ (\cite[Lem.~4.1(b)]{NZ13}). The corresponding oscillator groups $G_{A_n}$ are subgroups of $G_A$ and we also consider the corresponding complex groups $G_{A_n,\C}$ and complex semigroups $S_{A_n}$. We set 
$$G_{(\infty)}^\cO:=\bigcup_n G_{A_n}^\cO\ ,\ \ S_{(\infty)}:= \bigcup_n S_{A_n} \ \ \text{ and } \ \ G_{(\infty),\C}:=\Heis_\C((V_{(\infty)})_\C^\cO,\omega_\C) \rtimes_\gamma \C = \bigcup_n G_{A_n,\C}$$
which are dense in $G_A^\cO, S_A$ resp. $G_{A,\C}$ by Proposition \ref{Otopprop} (note $\bigcup_n V_n^\cO =(\bigcup_n V_n)^\cO$). Recall the polar map $\psi: G_A^\cO\times W_\infty^\cO \rightarrow S_A$. Proposition \ref{polarmap} applied to $G_{A_n}, n\in\N$, entails 
$$\psi(G_{(\infty)}^\cO\times (W_\infty^\cO\cap \g_{(\infty)}^\cO))=S_{(\infty)}=\{(z,x,s) \in G_{{(\infty)},\C} : \Im s > 0 \}.$$
Note that $S_{(\infty)}=G_{(\infty),\C}\cap S_A$ is open in $G_{(\infty),\C}$. Now consider a locally bounded map $f:S_{(\infty)}\rightarrow B$ into some Banach space $B$, where $S_{(\infty)}\subset S_A$ is equipped with the subspace topology. Then $f$ is holomorphic if and only it is holomorphic on affine discs and hence if and only $f\vert_{S_{A_n}}$ is holomorphic for all $n\in \N$.
\end{rmk}

\begin{thm}[Smoothness condition for semibounded representations]\label{extensionofdensesubgrepr}
Let $V_n\subset V=C^\infty(A)$ be $\gamma$-invariant closed complex subspaces of $V$ with $V_n \subset V_{n+1}$ for all $n\in \N$ and such that $\bigcup_n V_n$ is dense in $V$. Let $\pi^0:\Heis(\bigcup_n V_n) \rtimes_\gamma \R \rightarrow \U(\H)$ be a (not necessarily continuous) representation satisfying $\pi^0(t,0,0)=e^{it}\1$ for all $t\in\R$. Assume that $\pi^0\vert_{\Heis(V_n)\rtimes_\gamma\R}$ is a smooth representation for all $n\in \N$ and that the self-adjoint generator $B$ of $s\mapsto \pi^0(0,0,s)$ is bounded from below. Then $\pi^0$ is continuous and extends uniquely to a semibounded representation $\pi:G_A \rightarrow \U(\H)$.
\end{thm}
\begin{proof}
We set $V_{(\infty)}:=\bigcup_n V_n$ and $A_n:=A\vert_{V_n}$. Recall $G_{(\infty)}^\cO=\bigcup_n G_{A_n}^\cO$ and $S_{(\infty)}=\bigcup_n S_{A_n}$ from Remark \ref{complexlimitrmk}. From Proposition \ref{polarmap} we recall the polar map 
$$\psi: G_A^\cO\times W_\infty^\cO \rightarrow S_A, (g,w)\mapsto g\exp_\C(iw).$$
By Remark \ref{complexlimitrmk} we know that $\psi(G_{(\infty)}^\cO\times (W_\infty^\cO\cap \g_{(\infty)}^\cO))=S_{(\infty)}$. By Lemma \ref{lemma0} $\pi^0\vert_{G_{A_n}}$ is semibounded for all $n\in\N$ and we may define
$$\hat{\pi}^0:S_{(\infty)} \rightarrow B(\H),g\exp_\C(iw) \mapsto \pi^0(g) e^{i\dd\pi^0(w)}.$$
Let $a:=\inf(\Spec B)$. Consider
$$\alpha_a:S_A \rightarrow \R, g\exp_\C(iw) \mapsto e^{\frac{ \|x\|^2}{2s}-t- sa}\ \quad \text{ for $w=(t,x,s)\in W_\infty^\cO$}.$$
Note $\|\hat\pi^0(g\exp_\C(iw))\|=e^{s_{\pi^0}(w)}=\alpha_a(g\exp_\C(iw))$ (see equation \eqref{suppfuncspec}). Thus $\hat{\pi}^0\vert_{S_{A_n}}$ is $\alpha_a$-bounded and holomorphic for all $n\in \N$ (Theorem~\ref{sbrepholrepcorrespondenceG_A}). We conclude that $\hat{\pi}^0$ is $\alpha_a$-bounded and thus also holomorphic (cf. Remark~\ref{complexlimitrmk}). By Proposition \ref{polarmap}(c), the map $\alpha_a$ is continuous w.r.t. the $C^\infty$-topology on $S_A$. Hence $\hat{\pi}^0$ is holomorphic w.r.t. the $C^\infty$-topology on $S_{(\infty)}$. It follows that, for every $v\in \H$, the map
$$G_{(\infty)}^\cO \rightarrow \H, g \mapsto \pi^0(g)\hat\pi^0(0,0,i)v=\hat\pi^0(g\cdot(0,0,i))v$$
is smooth when $G_{(\infty)}^\cO \subset G_A$ is equipped with the subspace topology. As $\hat\pi^0(0,0,i)\H$ is dense in $\H$ (since $\ker e^{i\dd\pi^0(0,0,1)}=0$), we conclude that $\pi^0\vert_{G_{(\infty)}^\cO}$ is a smooth representation when $G_{(\infty)}^\cO \subset G_A$ is equipped with the subspace topology. With Corollary \ref{smoothextension}(a) we obtain a smooth representation $\pi:G_A \rightarrow \U(\H)$ extending $\pi^0$. By Lemma \ref{lemma0} this representation $\pi$ is semibounded.
\end{proof}

\begin{rmk}\label{extensionofdensesubgreprrmk}
The preceding theorem is quite useful. Suppose for instance that we have $V_n \subset V$ finite-dimensional $\gamma$-invariant complex subspaces with $V_n \subset V_{n+1}$ for all $n\in \N$, $\bigcup_n V_n \subset V$ dense and $\pi^H:\Heis(\bigcup_n V_n) \rightarrow \U(\H)$ a ray-continuous representation with $\pi^H(t,0,0)=e^{it}\1$. Let us further assume the existence of a continuous unitary one-parameter group $\tilde{\gamma}\colon \R \rightarrow \U(\H)$ with self-adjoint generator bounded from below such that
$$\tilde{\gamma}(t)\pi^H(x)\tilde{\gamma}(t)^{-1}=\pi^H(\gamma(t)x)$$
holds for all $x\in \bigcup_n V_n, t\in \R$. Then, with the preceding theorem, we obtain a unique semibounded representation $\pi: G_A \rightarrow \U(\H)$ with $\pi\vert_{\Heis(\bigcup_n V_n)}=\pi^H$ and $\pi(0,0,s) = \tilde{\gamma}(s)$. Therefore we obtain in this context a one-to-one correspondence between semibounded representations of $G_A$ on a Hilbert space $\H$ and ray-continuous representations of $\Heis(\bigcup_n V_n)$ on $\H$ together with an implementation of $\gamma$ by a unitary one-parameter group on $\H$ whose self-adjoint generator is bounded from below.
\end{rmk}

\appendix
\renewcommand{\thethmcount}{\Alph{section}.\arabic{thmcount}}
\renewcommand{\theequation}{\Alph{section}.\arabic{equation}}
\section{Appendix: Entire vectors for equicontinuous actions of $\R$}\label{entirevectorsapp}
In this section let $V$ be a locally convex space over $\mathbb{K}\in \{\R,\C\}$ and let $\gamma:\R\rightarrow \End(V)$ be a representation defining a smooth action $\R \times V \rightarrow V, (t,x)\mapsto \gamma(t)x$. We set $D:=\gamma'(0)\in \End(V)$. The set of continuous seminorms on $V$ is denoted by $\Gamma(V)$. By $\Gamma(V)^\gamma$ we denote the set of $\gamma$-invariant continuous seminorms on $V$.

\begin{dfn}\label{equicontdef}
The representation $\gamma$ is called \textit{equicontinuous} if the subset $\gamma(\R) \subset \End(V)$ is equicontinuous, i.e., if for every open $0$-neighborhood $U$ in $V$ there exists a $0$-neighborhood $W$ in $V$ such that $\gamma(t)W \subset U$ holds for every $t\in \R$.
\end{dfn}

\begin{prop}\label{equicontprop}{\rm(\cite[Prop. A.1]{Ne10c})}
The following conditions are equivalent:
\begin{itemize}
\item[\rm(a)] $\gamma$ is equicontinuous.
\item[\rm(b)] The topology on $V$ is defined by the set of $\gamma$-invariant seminorms $\Gamma(V)^\gamma$.
\item[\rm(c)] There exists a basis of $\gamma$-invariant absolutely convex $0$-neighborhoods in $V$.
\end{itemize}
\end{prop}

\begin{dfn}\label{equicontholcvectdef}
For each $q\in \Gamma(V)$ on $V$ we define
$$q_n(v):= \sum_{k\geq0} \frac{1}{k!} n^k  q(D^kv)\in [0,\infty]$$
for $v\in V, n\in \N$. Note $q(v)\leq q_n(v)$ for all $n\in \N, v\in V$. We consider the subspace
$$V^\cO:= \{ v\in V : q_n(v) <\infty \text{ for all $q\in \Gamma(V)$ and $n\in \N$} \}.$$
The elements of $V^\cO$ are called \textit{$\gamma$-holomorphic vectors} in $V$. We equip $V^\cO$ with its natural \textit{$C^\cO$-topology} given by the seminorms $q_n\vert_{V^\cO}$, where $(q,n)\in \Gamma(V)\times \N$. Note that $V^\cO$ is invariant under $\gamma(t), t\in \R$ and $D$.
\end{dfn}
\begin{rmk}\label{equicontrmk0}
\begin{itemize}
\item[\rm(a)] Suppose that $V$ is sequentially complete and let $v\in V$. Then $v\in V^\cO$ holds if and only if the orbit map $t\mapsto \gamma(t)v$ extends to a holomorphic map $\C\rightarrow V_\C$.
\item[\rm(b)] If $\Gamma(V)$ is replaced in Definition~\ref{equicontholcvectdef} by any set of seminorms generating the topology on $V$, the space $V^\cO$ and the $C^\cO$-topology remain the same. In particular, if $\gamma$ is equicontinuous then $V^\cO= \{ v\in V : q_n(v) <\infty \text{ for all $q\in \Gamma(V)^\gamma$ and $n\in \N$} \}$ and the $C^\cO$-topology is generated by the $\gamma$-invariant seminorms $q_n\vert_{V^\cO}$ where $(q,n)\in \Gamma(V)^\gamma\times \N$, cf. Proposition~\ref{equicontprop}.
\item[\rm(c)] If $\gamma$ is equicontinuous, then $\gamma(\R)\vert_{V^\cO} \subset \End(V^\cO)$ is also equicontinuous with $V^\cO$ equipped with the $C^\cO$-topology. This follows from (b) and Proposition~\ref{equicontprop}.
\end{itemize}
\end{rmk}

\begin{rmk}\label{equicontcomplexificationrmk}
Suppose that $V$ is a real locally convex space. Let $V_\C=V\otimes_\R \C$ be the complexification of $V$ with its natural topology turning it into a complex locally convex space. By complex-linear extension of $\gamma(t)$ we obtain a one-parameter group $\gamma_c: \R \rightarrow \End(V_\C)$ with generator $\gamma_c'(0)=D_\C$ inducing a smooth action $\R\times V_\C \rightarrow V_\C, (t,v)\mapsto \gamma_c(t)v$. Then $\gamma$ is equicontinuous if and only if $\gamma_c$ is equicontinuous. Furthermore $(V_\C)^\cO=(V^\cO)_\C$ holds as topological vector spaces, where $(V_\C)^\cO$ denotes the space of $\gamma_c$-holomorphic vectors in $V_\C$.
\end{rmk}

\begin{prop}\label{equicontholvecprop}
Suppose that $V$ is sequentially complete. Then the following assertions hold:
\begin{itemize}
\item[\rm(a)] If $\gamma$ is equicontinuous, then $V^\cO \subset V$ is a dense subspace.
\item[\rm(b)] The space $V^\cO$ equipped with the $C^\cO$-topology is sequentially complete.
\item[\rm(c)] Suppose, in addition, that $V$ is a complex locally convex space. Then
$$\gamma_\C: \C \times V^\cO \rightarrow V^\cO,(z,v)\mapsto \gamma_\C(z)v:= \sum_{k\geq0} \frac{1}{k!} z^k D^k v$$
defines a holomorphic action extending the action of $\gamma$ on $V^\cO$. In particular, the map $\R \times V^\cO \rightarrow V^\cO,(t,v)\mapsto \gamma(t)v$ is analytic.
\end{itemize}
\end{prop}
\begin{proof}
For (a) we may assume in view of Remark~\ref{equicontcomplexificationrmk} that $\mathbb{K}=\C$. Now (a) follows from \cite[Prop 1.13]{Ma92}.\newline\newline
(b) Let $(v_n), v_n\in V^{\cO}$ be a Cauchy sequence for all $q_m, q\in \Gamma(V),m\in \N$. Then $(v_n)$ is also a Cauchy sequence for every $q\in \Gamma(V)$. Since $V$ is sequentially complete we find $v\in V$ such that $v_n\rightarrow v$ holds in $V$. Since $D:V\rightarrow V$ is continuous this implies $D^k v_n\rightarrow D^k v$ in $V$ as $n\rightarrow \infty$ for every $k\in \N_0$. Now let $q\in \Gamma(V)$ and $m\in\N$. For $\varepsilon>0$ there is an $n_0\in \N$ such that $q_{2m}(v_n-v_{n'})\leq \frac{\varepsilon}{2}$ holds for all $n,n'\geq n_0$. Hence
$$\frac{1}{k!}q(D^k(v_n-v_{n'}))m^k \leq \frac{\varepsilon}{2^{k+1}} \quad \text{ for all $n,n'\geq n_0, k\geq 0$}.$$
If we let $n'\rightarrow \infty$ we obtain $\frac{1}{k!}q(D^k(v_n-v))m^k \leq \frac{\varepsilon}{2^{k+1}}$ for all $n \geq n_0, k \geq 0$. This yields $$q_m(v_n-v)\leq \eps \quad \text{ for all $n\geq n_0$.} $$
In particular this yields $q_m(v)\leq q_m(v-v_{n_0})+q_m(v_{n_0}) <\infty$, so that $v\in V^\cO$. It also implies $v_n\rightarrow v$ with respect to the $C^{\cO}$-topology.\newline\newline
(c) First note that $\gamma_\C$ is well-defined as a map $\C \times V^\cO \rightarrow V$, since for every $v\in V^\cO$ the power series $\sum_{k\geq0} \frac{1}{k!} z^k D^k v$ converges absolutely in $V$ for all $z\in \C$. For $m\in \N, z\in \C, v\in V^{\cO}$ we calculate:
\begin{align*} 
q_m(\gamma_{\C}(z)v) & = \sum_{k\geq0} \frac{1}{k!}m^k q\Big(D^k\Big(\sum_{l\geq0} \frac{1}{l!}z^l D^l v\Big)\Big) \leq \sum_{k\geq 0} \sum_{l \geq 0} \frac{1}{k!}\frac{1}{l!}m^k|z|^l q(D^{k+l}v) \\
&=\sum_{k\geq 0} \sum_{l \geq 0} \frac{1}{(k+l)!} \binom{k+l}{k}m^k|z|^l q(D^{k+l}v) \\ 
&= \sum_{n\geq 0}\frac{1}{n!}(m+|z|)^nq(D^nv) \leq q_{m+c}(v)
\end{align*}
for $c\in \N$ with $c\geq |z|$. This shows that $\gamma_{\C}(z)v\in V^{\cO}$ and that $\gamma_\C$ is locally bounded. Certainly the map $\C\rightarrow V^\cO, z\mapsto \gamma_{\C}(z)v$ is holomorphic for each $v\in V^{\cO}$. Moreover $\gamma_\C(z)v$ is linear and holomorphic in $v$ for fixed $z$. Since $\gamma_\C$ is locally bounded we obtain that $\gamma_\C$ is holomorphic. A direct calculation shows that $\gamma_\C$ is indeed an action of $(\C,+)$ on $V^\cO$. In particular we may consider $\gamma_\C$ as a representation $\gamma_\C:\C \rightarrow \End(V^\cO)$. For $v\in V^\cO$ consider the map $f_v:\R \rightarrow V, t \mapsto \gamma_\C(t)v$. We have $f_v(0)=v$ and $f_v'(t)=Df_v(t)$. As the generator $D$ determines $\gamma$, cf. \cite[Thm. 1.5]{Ma92}, we conclude that $\gamma_\C(t)v=f_v(t)=\gamma(t)v$, i.e., $\gamma_\C$ extends the action of $\gamma$ on $V^\cO$.
\end{proof}

\begin{rmk}\label{equicontsomermks}
\begin{itemize}
\item[\rm(a)] If $V$ is complete then $V^\cO$ is also complete. This may be proven as Proposition~\ref{equicontholvecprop}(b) by considering nets instead of sequences.
\item[\rm(b)] Suppose that $V$ is complete and assume that $\gamma$ is equicontinuous. Then $\gamma(\R)\vert_{V^\cO} \subset \End(V^\cO)$ is also equicontinuous by Remark~\ref{equicontrmk0}(c). For every $f\in L^1(\R,\mathbb{K})$, where $\mathbb{K}\in \{\R,\C\}$, the continuous linear operator $\pi(f):=\int_\R f(t) \gamma(t)\ dt\in \End(V^\cO)$ exists in the sense of weak integrals. Moreover this leads to an algebra representation 
$$\pi: (L^1(\R,\mathbb{K}),+,*) \rightarrow B(V^\cO),$$
where $*$ denotes the convolution product in the group algebra $L^1(\R,\mathbb{K})$. For this we refer to \cite[Def. A.5(a)]{Ne10c}.
\end{itemize}
\end{rmk}

\begin{rmk}\label{equiconthilbertrmk}
Consider the situation when $\gamma:\R\rightarrow \U(H)$ is a unitary one-parameter group on a complex Hilbert space $H$ with self-adjoint generator $A$. With $V:=C^\infty(A)$ the space of $\gamma$-smooth vectors in $H$ equipped with its natural $C^\infty$-topology we obtain a smooth action $\R \times V \rightarrow V,(t,v)\mapsto \gamma(t)v$. Obviously, $\gamma(\R)\subset \End(V)$ is equicontinuous and $V$ is complete. Define $q_n(v):=\sum_{k\geq0}\frac{n^k}{k!} \|A^kv\|$ for $v\in V,n\in \N$. Then it is easily seen that $V^\cO =\{ v \in V : q_n(v) <\infty \ \forall n\in \N\}$ and that the $C^\cO$-topology on $V^\cO$ is given by the norms $q_n\vert_{V^\cO}, n\in \N$.
\end{rmk}

In the following, let $V:=C^\infty(A), H$ and $\gamma$ be as in Remark~\ref{equiconthilbertrmk}. By $B(H)^\gamma$ we denote the subspace of operators in $B(H)$ which commute with $\gamma(t)$ for all $t\in \R$. We equip both $B(H)^\gamma$ and $B(H)$ with the operator norm topology induced by the norm $\|\cdot\|$ on~$H$. It turns them both into Banach spaces. Set $D:=iA\vert_V$ and assume $\ker D=0$.

\begin{lem}\label{equiconthollem1}
\begin{itemize}
\item[\rm(a)] Every $T\in B(H)^\gamma$ leaves $V^\cO$ invariant, i.e. $T(V^\cO)\subset V^\cO$.
\item[\rm(b)] The map $\alpha : B(H)^\gamma \times V^\cO \rightarrow V^\cO, (T,v)\mapsto Tv$ is holomorphic. Moreover, for every $k\in \N_0$, the map $\alpha$ is holomorphic when $V^\cO$ is equipped with the norm $x\mapsto \|A^kx\|$. In particular $\alpha$ is holomorphic when $V^\cO$ is equipped with the $C^\infty$-topology.
\end{itemize}
\end{lem}
\begin{proof}
Recall the seminorms $q_n, n\in \N$ from Remark~\ref{equiconthilbertrmk}. Let $T\in B(H)^\gamma$. Certainly $T$ leaves $V=C^\infty(A)$ invariant. For $v\in V^\cO$ we have
$$q_n(Tv)=\sum_{k\geq0}\frac{n^k}{k!} \|A^kTv\|=\sum_{k\geq0}\frac{n^k}{k!} \|TA^kv\| \leq \|T\|\cdot q_n(v).$$
This shows (a).\newline
(b) Since $\alpha$ is complex bilinear, we only have to show that $\alpha$ is continuous w.r.t. the stated topologies. This follows from the following estimates:
$$q_n(\alpha(T,v)) =q_n(Tv) \leq \|T\|\cdot q_n(v)$$
for $n\in \N, v\in V^\cO$ and
$$\|A^k\alpha(T,v)\| = \|T A^kv\| \leq \|T\| \cdot \|A^kv\|$$
for $k\in \N_0, v\in V$.
\end{proof}

Let $U\subset \C$ be an open subset and $f:U\times \R\rightarrow \C$ be a map such that $f(s,\cdot)$ is a bounded measurable function on $\R$ for every $s\in U$. We set $f_s(x):=f(s,x)$ and define $f(s,A):=f_s(A)\in B(H)$ by functional calculus of the self-adjoint operator $A$. Note that $f_s(A)\in B(H)^\gamma$.

\begin{lem}\label{equiconthollem2}
Assume that $f(\cdot,x)$ is a holomorphic function on $U$ for every $x\in \R$. Assume further that $\frac{\partial f}{\partial s}(s,x)$ is a bounded function in $x$ locally uniformly in $s$. Then the map $F:U \rightarrow B(H)^\gamma, s\mapsto f(s,A)$ is holomorphic.
\end{lem}
\begin{proof}
The assumption on $\frac{\partial f}{\partial s}$ implies that the functions $f(s,\cdot)$ are bounded locally uniformly in $s$. Hence $F$ is locally bounded. Let $s\in U$ and let $s_n\rightarrow s$ with $s_n\neq s$ for all $n\in \N$. Consider
$$g_n(x):= \frac{f(s_n,x)-f(s,x)}{s_n-s} - \frac{\partial f}{\partial s}\big(s,x\big)$$
and note $g_n(x)\rightarrow 0$ for all $x\in \R$. The assumption on $\frac{\partial f}{\partial s}$ implies that the functions $g_n$ are uniformly bounded. Hence \cite[Thm. VII.2(d)]{RS73} implies that $g_n(A)\rightarrow 0$ strongly. Thus, the map $U \rightarrow H, s\mapsto f(s,A)v$ is complex differentiable, i.e. $U \rightarrow H, s\mapsto F(s)v$ is holomorphic for every $v\in H$. Since $F$ is locally bounded we conclude that $F$ is holomorphic (cf. \cite[Lem. 3.4]{Ne08}).
\end{proof}

\begin{prop}\label{Otopprop}
The space $V^b:=\{v\in V : \text{ $v$ is $\gamma$-bounded}\}$ is dense in $V^{\cO}$. More generally, if $L\subset V$ is dense w.r.t. the $\|\cdot\|$-topology and invariant under $P_A(J)$ for all bounded intervals $J\subset\R$, then $L\cap V^\cO$ is dense in $V^\cO$. Here $P_A$ denotes the spectral measure of $A$. 
\end{prop}
\begin{proof}
For $v\in V^{\cO}$ define $v_n:=P_A([-n,n])v\in V^b$. Let $m\in \N$. Then
$$q_m(v-v_n)\leq\sum_{k=0}^N\frac{m^k}{k!}\|(1-P_A([-n,n]))A^kv\| +r(N)$$
where $r(N):=\sum_{k>N}\frac{m^k}{k!}\|A^kv\|$ converges to zero for $N\rightarrow \infty$. For every $k$ we have:
$$(1-P_A([-n,n]))A^kv \rightarrow 0 \quad \text{ for $n\rightarrow \infty$}.$$ 
We conclude that $q_m(v-v_n)\rightarrow 0$ for $n\rightarrow \infty$. Thus $V^b\subset V^\cO$ is dense. For the second statement we only have to show that $L\cap V^b$ is dense in $V^b$. Let $v\in V^b$ and choose a bounded interval $J\subset\R$ with $P_A(J)v=v$. We choose $v_n\in L$ with $v_n\rightarrow v$ in the norm topology. Then $v'_n:=P(J)v_n\in L\cap V^b$ and obviously $v'_n\rightarrow v$ in the $C^\cO$-topology.
\end{proof}

\subsection*{Acknowledgements.}

I am grateful to Karl-Hermann Neeb for helpful and inspiring discussions and for reading preliminary versions of the manuscript. I acknowledge the support of DFG-grant NE 413/7-2 in the framework of SPP “Representation Theory”.

\end{document}